\documentclass{amsart}
\usepackage{amsmath,amsxtra,amssymb}
\usepackage[all]{xy}

\newcommand{\hz}{\vspace{0.5cm}}

\renewcommand{\qed}{\hspace*{\fill}$\Box$\hz\pagebreak[1]}

\newtheorem{theorem}{Theorem}[section]
\newtheorem{cor}[theorem]{Corollary}
\newtheorem{prop}[theorem]{Proposition}
\newtheorem{rem}[theorem]{Remark}
\newtheorem{lemma}[theorem]{Lemma}
\newtheorem{defi}[theorem]{Definition}
\newtheorem{exam}[theorem]{Example}

\begin{document}
\title{Factorization and dilation problems for completely positive maps on
von Neumann algebras}
\author{Uffe Haagerup$^{(1)}$ and Magdalena Musat$^{(2)}$}
\address{$^{(1)}$ Department of Mathematical Sciences, University of
Copenhagen, Universitetsparken 5, 2100 Copenhagen {\O}, Denmark.\\
$^{(2)}$ Department of Mathematical Sciences, University of
Copenhagen, Universitetsparken 5, 2100 Copenhagen {\O}, Denmark.}
\email{$^{(1)}$haagerup@math.ku.dk\\$^{(2)}$musat@math.ku.dk}.

\date{}

\footnotetext {$^{(1)}$ Partially supported by the Danish Natural
Science Research Council.\\ \hspace*{0.48cm} $^{(2)}$ Partially
supported by the National Science Foundation, DMS-0703869.}


\maketitle

\begin{abstract}
We study factorization and
dilation properties of Markov maps between von Neumann algebras equipped with normal faithful states, i.e., completely positive unital maps which preserve the given states and also intertwine
their automorphism groups. The starting point for our investigation has been
the question of existence of
non-factorizable Markov maps,
as formulated
by C. Anantharaman-Delaroche.
We provide simple examples of
non-factorizable Markov maps on $M_n(\mathbb{C})$ for
all $n\geq 3$\,, as well as an example of a one-parameter semigroup
$(T(t))_{t\geq 0}$ of Markov maps on $M_4(\mathbb{C})$ such
that $T(t)$ fails to be factorizable for all small
values of $t > 0$\,. As applications, we solve in the negative
an open problem in quantum information theory
concerning an asymptotic version of the quantum Birkhoff conjecture, as well as
we sharpen the existing lower bound estimate for
the best constant in the noncommutative little Grothendieck
inequality.

\end{abstract}

\section{Introduction}
\setcounter{equation}{0}

Motivated by the study of ergodic actions of free groups on noncommutative spaces, C. Anantharaman-Delaroche investigated in \cite{AD} the noncommutative analogue of G.-C. Rota's "Alternierende Verfahren" theorem from classical probability, asserting that if $T$ is a measure-preserving Markov operator on the probability space $(\Omega, \mu)$, then for every $p\geq 1$ and $f\in L^p(\Omega, \mu)$\,, the sequence $T^n(T^*)^n (f)$ converges almost everywhere, as $n\rightarrow \infty$\,. In the noncommutative setting, the probability space is replaced by a von Neumann algebra $M$, equipped with a normal faithful state $\phi$, and $T$ is now a unital, completely positive map on $M$ such that $\phi\circ T=\phi$\,. However, in this setting the existence of the adjoint map $T^*$ is not automatic. It turns out (see more precise references below) that it is equivalent to the fact that $T$ commutes with the modular automorphism group of $\phi$\,. This motivated the following definition considered in \cite{AD} (cf. Definition 2.6), where we have chosen to follow the slightly modified notation from \cite{Ric}:

\begin{defi}\label{defmarkov}
Let $(M\,, \phi)$ and $(N\,, \psi)$ be von Neumann algebras
equipped with normal faithful states $\phi$ and $\psi$\,,
respectively. A linear map $T\colon M\rightarrow N$ is called a $(\phi, \psi)$-Markov map if
\begin{enumerate}
\item [$(1)$] $T$ is completely positive
\item [$(2)$] $T$ is unital
\item [$(3)$] $\psi\circ T=\phi$
\item [$(4)$] $T\circ \sigma_t^{\phi}=\sigma_t^\psi\circ T$\,, for all $t\in \mathbb{R}$\,, where $(\sigma_t^{\phi})_{t\in \mathbb{R}}$ and $(\sigma_t^{\psi})_{t\in \mathbb{R}}$ denote the automorphism goups of the states $\phi$ and $\psi$\,, respectively.
\end{enumerate}
In particular, when $(M\,, \phi)=(N\,, \psi)$\,, we say that $T$ is a $\phi$-Markov map.
\end{defi}

Note that a linear map $T\colon M\rightarrow N$ satisfying conditions $(1)-(3)$ above is automatically normal. If, moreover, condition $(4)$ is satisfied, then
it was proved in \cite{AC} (see also Lemma 2.5 in \cite{AD}) that there exists a unique completely positive, unital map $T^*\colon N
\rightarrow M$ such that
\begin{equation}\label{eq5000567}\phi(T^*(y)x) =\psi(y T(x))\,, \quad x\in M\,, y\in N\,. \end{equation}
It is easy to show that $T^*$ is a $(\psi\,, \phi)$-Markov map.

\begin{rem}\label{rem78787777754}\rm
A special case of interest is the one of a $(\phi, \psi)$-Markov map $J\colon M\rightarrow N$ which is a $*$-monomorphism. In this case
$J(M)$ is a $\sigma^\psi$-invariant sub-von Neumann algebra of $N$. Hence, by \cite{Ta1}, there is a unique $\psi$-preserving normal faithful conditional expectation $\mathbb{E}_{J(M)}$ of $N$ onto $J(M)$, and thus the adjoint $J^*$ is given by $J^*=J^{-1}\circ \mathbb{E}_{J(M)}$\,.
\end{rem}

C. Anantharaman-Delaroche proved in \cite{AD} that the noncommutative analogue of Rota's theorem holds for Markov maps which are factorizable in the following sense (cf. Definition 6.2 in \cite{AD}):
\begin{defi}\label{defi898}
A $(\phi, \psi)$-Markov map $T\colon M\rightarrow N$  is called factorizable
if there exists a von Neumann algebra $P$ equipped with a faithful
normal state $\chi$\,, and $*$-monomorphisms $J_0\colon M\rightarrow P$ and $J_1\colon N\rightarrow
P$ such that $J_0$ is $(\phi, \chi)$-Markov and $J_1$ is $(\psi, \chi)$-Markov, satisfying, moreover, $T=J_0^*\circ J_1$\,.
\end{defi}

\begin{rem}\label{rem787877777}\rm

$(a)$ Note that if both
$\phi$ and $\psi$ are {\em tracial} states on $M$ and $N$,
respectively, and $T\colon M\rightarrow N$ is
factorizable, then the factorization can be chosen through a von Neumann algebra with a faithful normal {\em tracial} state, as well. This
can be achieved by replacing $(P, \chi)$ whose existence is ensured by the definition of factorizability by $(P_\chi,
\chi_{\vert_{P_\chi}})$\,, where $P_\chi$ denotes the centralizer of the state
$\chi$\,, since $J_0(M)\subseteq P_\chi$ and $J_1(N)\subseteq P_\chi$\,.

$(b)$ The class of factorizable $(\phi, \psi)$-Markov maps is known to be closed under composition, the adjoint operation, taking convex combinations and $w^*$-limits (see Proposition 2 in \cite{Ric}).
\end{rem}

C. Anantharaman-Delaroche raised in \cite{AD} the question whether all Markov maps are factorizable.
This was the starting point of investigation for our paper.
The class of maps which are known to be factorizable includes all Markov maps between abelian von Neumann algebras (as it was explained in \cite{AD}, Remark 6.3 $(a)$), the trace-preserving Markov maps on $M_2(\mathbb{C})$ (due to a result of B. K\"ummerer from \cite{Ku2}), as well as Schur multipliers associated to positive semi-definite real matrices having diagonal entries all equal to 1 (as shown by E. Ricard in \cite{Ric}).
It is therefore natural to further study the problem of factorizability of $\tau_n$-Markov maps on $M_n(\mathbb{C})$\,, for $n\geq 3$\,, where $\tau_n$ denotes the unique normalized trace on the $n\times n$ complex matrices.

In Section 2 we give a general characterization of factorizable $\tau_n$-Markov maps on $M_n(\mathbb{C})$, as well as a characterization of those $\tau_n$-Markov maps which lie in the convex hull of $*$-automorphisms of $M_n(\mathbb{C})$\,. We also discuss the case of Schur multipliers.

As an application, we construct in Section 3 several examples of non-factorizable $\tau_n$-Markov maps on $M_n(\mathbb{C})$\,, $n\geq 3$ (cf. Examples \ref{exp1} and \ref{exp2}), an example of a factorizable Schur multiplier on $(M_6(\mathbb{C})\,, \tau_6)$ which does not lie in the convex hull of $*$-automorphisms of $M_6(\mathbb{C})$ (cf. Example \ref{exp3}), as well as an example of a one-parameter semigroup
$(T(t))_{t\geq 0}$ of $\tau_4$-Markov maps on $M_4(\mathbb{C})$ such that $T(t)$ fails to be factorizable for all small values of $t > 0$ (see Theorem \ref{th990})\,.
This latter example is to be contrasted with a result of B. K\"ummerer and H. Maasen from \cite{KM}, asserting that if $(T(t))_{t\geq 0}$ is a one-parameter semigroup of $\tau_n$-Markov maps on $M_n(\mathbb{C})$\,, $n\geq 1$\,, such that each $T(t)$ is self-adjoint, then $T(t)$ is factorizable, for all $t\geq 0$\,. We have been informed of recent work of M. Junge, E. Ricard and D. Shlyakhtenko, where they have generalized K\"ummerer and Maasen's result to the case of a strongly continuous one-parameter semigroup of self-adjoint Markov maps on an arbitrary finite von Neumann algebra.  This result has been independently obtained by Y. Dabrowski (see the preprint \cite{Dab}).

In Section 4 we discuss the connection between the notion of factorizability and K\"ummerer's notions of dilation, respectively, of Markov dilation, that he introduced in \cite{Ku3}. The starting point for our analysis was a private communication by C. Koestler \cite{Ko}, who informed us in the Spring of 2008 that for a $\phi$-Markov map on a von Neumann algebra $M$, factorizability is equivalent to the existence of a dilation (in the sense of \cite{Ku3})\,, and that  K\"ummerer in his unpublished Habilitationsschrift \cite{Ku1} had constructed examples of $\tau_n$-Markov maps on $M_n(\mathbb{C})$\,, $n\geq 3$\,, having no dilations, and hence being non-factorizable.  The equivalence between factorizability and the existence of a dilation is based on an inductive limit argument also from K\"ummerer's unpublished work \cite{Ku1}. In Theorem  \ref{th77777777232323} and its proof, we provide the details of the argument communicated to us by C. Koestler. Moreover, we show that the existence of a dilation is equivalent to the--seemingly stronger--condition of existence of a Markov dilation in the sense of K\"ummerer.

Section 5 is devoted to the study of the so-called {\em Rota dilation property} of a Markov map, introduced by M. Junge, C. LeMerdy and Q. Xu in \cite{JMX}. This notion has proven to be very fruitful for
the development of semigroup theory in the noncommutative setting, and applications to noncommutative $L_p$-spaces and noncommutative harmonic analysis (see, e.g., \cite{Me}\,, \cite{JMe}). The Rota dilation property of a Markov map implies its factorizability, but it is more restrictive, as it forces the map to be self-adjoint. As a consequence of Theorem 6.6 in \cite{AD}, the square of any factorizable self-adjoint Markov map has the Rota dilation property.  Our main result in this section is that
there exists a self-adjoint $\tau_n$-Markov map $T$ on $M_n(\mathbb{C})$\,, for some positive integer $n$\,, such that $T^2$ does not have the Rota dilation property
(see Theorem \ref{th5656}), and therefore the analogue of Rota's classical dilation theorem for Markov operators does not hold, in general, in the noncommutative setting.

The existence of non-factorizable Markov maps turned out to have an interesting application to an open problem in quantum information theory, known as {\em the asymptotic quantum Birkhoff conjecture}. The conjecture, originating in joint work of A. Winter, J. A. Smolin and
F. Verstraete (cf. \cite{SVW}), asserts that if $T\colon M_n(\mathbb{C})\rightarrow M_n(\mathbb{C})$ is a $\tau_n$-Markov map, $n\geq 1$\,, then $T$ satisfies the following {\em asymptotic quantum Birkhoff property}:
\begin{equation*}\lim\limits_{k\rightarrow \infty} d_{\text{cb}}\bigg({\textstyle{\bigotimes\limits_{i=1}^k}} \,T\,, \text{conv}(\text{Aut}({\textstyle{\bigotimes\limits_{i=1}^k}} M_n(\mathbb{C})))\bigg)=0\,. \end{equation*}
We would like to thank V. Paulsen for bringing this problem to our attention.
In Section 6 we solve the conjecture in the negative (see Theorem \ref{thasybirkh})), by showing that every non-factorizable $\tau_n$-Markov map on $M_n(\mathbb{C})$\,, $n\geq 3$\,, fails the above asymptotic quantum Birkhoff property.

Finally, in Section 7, as an application of some of the techniques developed in the previous sections, we prove that the best constant in the noncommutative little Grothendieck inequality (cf. \cite{PiSh} and \cite{HM}) is strictly greater than 1, thus sharpening the existing lower bound estimate for it.

\section{Factorizability of $\tau_n$-Markov maps on $M_n(\mathbb{C})$} \setcounter{equation}{0}

Let $P$ be a von Neumann algebra and $n$ a positive integer.
Recall (see, e.g., \cite{KR} (Vol. II, Sect. 6.6) that a family
$(f_{ij})_{1\leq i, j\leq n}$ of elements of $P$ is a set of {\em matrix units} in $P$ if
the following conditions are satisfied: $f_{i j}f_{k l}=\delta_{jk}f_{il}$\,, $1\leq i,
j, k, l\leq n$\,, $f_{ij}^*=f_{ji}$\,, $1\leq
i, j\leq n$ and $\sum_{i=1}^n f_{ii}=1_P$\,. If this is the case, then
$F\colon =\text{Span}\{f_{ij}: 1\leq i, j\leq
n\}$ is a $*$-subalgebra of $P$ isomorphic to $M_n(\mathbb{C})$ and
$1_P\in F$\,.

The following result is well-known, but we include a proof for the
convenience of the reader.

\begin{lemma}\label{lem1}
Let $P$ be a von Neumann algebra, $n$ a positive integer, and
$(f_{ij})_{1\leq i, j\leq n}$\,, $(g_{ij})_{1\leq i, j\leq n}$
two sets of matrix units in $P$\,. Then there exists a unitary
operator $u\in P$ such that
\begin{eqnarray*}
uf_{ij}u^*=g_{ij}\,, \quad 1\leq i, j\leq n\,.
\end{eqnarray*}
\end{lemma}

\begin{proof}
By hypothesis, $(f_{ii})_{1\leq i\leq n}$
and $(g_{ii})_{1\leq i\leq n}$ are two sets of pairwise orthogonal
projections in $P$ with $\sum_{i=1}^n
f_{ii}=1_P=\sum_{i=1}^n
g_{ii}$\,, satisfying, moreover, $f_{11}\sim f_{22}\sim\ldots \sim f_{nn}$ and $g_{11}\sim g_{22}\sim\ldots \sim g_{nn}$\,, respectively, where $\sim$ denotes the relation of equivalence of projections. By, e.g., \cite{KR} (Vol. II, Ex.
6.9.14), it follows that $f_{11}\sim g_{11}$\,, i.e., there exists a partial
isometry $v\in P$ such that $v^*v=f_{11}$ and $vv^*=g_{11}$\,.
Set now $u\colon =\sum_{i=1}^n g_{i1}vf_{1i}$\,. It is
elementary to check that $u$ is a unitary in $P$\,.
Moreover, for $1\leq k, l\leq n$\,, $uf_{kl}u^*=\sum_{i, j=1}^n
g_{i1}vf_{1i}f_{kl}f_{j1}v^*g_{1j}=g_{k1}vf_{11}v^*f_{1l}=g_{k1}g_{11}g_{1l}=g_{kl}$\,,
which proves the result.
\end{proof}

By a result of M.-D. Choi (see \cite{Ch}), a linear map $T\colon M_n(\mathbb{C})\rightarrow M_n(\mathbb{C})$ is completely positive if and only if $T$ can be written in the form
\begin{equation}\label{eq1}
Tx= \sum\limits_{i=1}^d a_i^*xa_i\,, \quad x\in M_n(\mathbb{C})\,,
\end{equation}
for some $a_1\,, \ldots\,, a_d\in M_n(\mathbb{C})$\,. The condition that $T$ is unital is then equivalent to $\sum_{i=1}^d a_i^*a_i=1_n$\,, while the condition that $T$ is trace-preserving, i.e., $\tau_n\circ T=\tau_n$, is equivalent to $\sum_{i=1}^d
a_ia_i^*=1_n$. Here $1_n$ denotes the identity matrix in $M_n(\mathbb{C})$\,.

\begin{theorem}\label{th1}
Let $T: M_n(\mathbb{C})\rightarrow M_n(\mathbb{C})$ be a
$\tau_n$-Markov map, written in the form (\ref{eq1})
where $a_1\,, \ldots\,, a_d\in M_n(\mathbb{C})$ are chosen to be linearly
independent and satisfy $\sum_{i=1}^d a_i^*a_i=\sum_{i=1}^d
a_ia_i^*=1_n$. Then the following conditions are
equivalent:
\begin{itemize}
\item [$(i)$] $T$ is factorizable.
\item [$(ii)$] There exists a finite von Neumann algebra $N$ equipped
with a normal faithful tracial state $\tau_N$ and a unitary operator
$u\in M_n(N)=M_n(\mathbb{C})\otimes N$ such that
\begin{equation}\label{eq2}
Tx=(\text{id}_{M_n(\mathbb{C})}\otimes \tau_N)(u^*(x\otimes
1_N)u)\, , \quad x\in M_n(\mathbb{C})\,. \end{equation}
\item [$(iii)$] There exists a finite von Neumann algebra $N$ equipped
with a normal faithful tracial state $\tau_N$ and $v_1\,, \ldots \,,
v_d\in N$ such that $u\colon = \sum_{i=1}^d
a_i\otimes v_i$ is a unitary operator in
$M_n(\mathbb{C})\otimes N$ and
\begin{eqnarray*}
\tau_N(v_i^*v_j)= \delta_{ij}\,, \quad 1\leq i, j\leq d\,.
\end{eqnarray*}
\end{itemize}
\end{theorem}

\begin{proof}
We first show that $(i)\Rightarrow (ii)$\,. Assume that $T$ is
factorizable. By Remark \ref{rem787877777} $(a)$, there exists a finite von
Neumann algebra $P$ with a normal faithful tracial state $\tau_P$
and two unital $*$-monomorphisms $\alpha\,,
\beta:M_n(\mathbb{C})\rightarrow P$ such that $T=\beta^*\circ
\alpha$\,. Note that $\alpha$ and $\beta$ are automatically
$(\tau_n\,, \tau_P)$-Markov maps, since $\tau_n$ is the unique normalized trace
on $M_n(\mathbb{C})$\,. Let $\{e_{ij}\}_{1\leq i, j\leq n}$ be the
standard matrix units in $M_n(\mathbb{C})$ and set
$f_{ij}\colon =\alpha(e_{ij})$\,, respectively, $g_{ij}\colon =
\beta(e_{ij})$\,, for all $1\leq i, j\leq n$\,. Choose now a unitary
operator $u\in P$ as in Lemma \ref{lem1}\,. Then $\beta(x)= u\alpha(x)u^*$\,, for all $x\in M_n(\mathbb{C})$\,.
Equivalently, $\alpha(x)=u^*\beta(x)u$\,, for all $x\in
M_n(\mathbb{C})$\,. Consider now the relative commutant
\[ N\colon =(\beta(M_n(\mathbb{C})))^\prime \cap P=\{g_{ij}: 1\leq
i, j\leq n\}^\prime \cap P\,, \] and let $\tau_N$ be the restriction
of $\tau_P$ to $N$\,. Since the map $\sum_{i, j=1}^n e_{ij}\otimes
x_{ij}\mapsto \sum_{i,j=1}^n g_{ij}x_{ij}$\,, where
$x_{ij}\in N$\,, $1\leq i, j\leq n$\,, defines a $*$-isomorphism of $M_n(\mathbb{C})\otimes
N$ onto $P$ (see, e.g., \cite{KR}, Vol. II, Sect. 6.6.), we can make the identifications
$P=M_n(\mathbb{C})\otimes N$\,, $\tau_P=\tau_n\otimes \tau_N$ and $\beta(x)=x\otimes 1_{N}$\,, $x\in M_n(\mathbb{C})$\,. This
implies that $\alpha(x)=u^*(x\otimes 1_N)u$\,, $x\in
M_n(\mathbb{C})$\,. Since $T=\beta^*\circ \alpha=\beta^{-1}\circ
\mathbb{E}_{\beta(M_n(\mathbb{C}))}\circ \alpha$ (see Remark \ref{rem78787777754}), then
\[ Tx\otimes
1_N=\mathbb{E}_{M_n(\mathbb{C})\otimes
{1_N}}(u^*(x\otimes 1_N)u)\,, \quad x\in M_n(\mathbb{C})\,, \] where
$\mathbb{E}_{M_n(\mathbb{C})\otimes {1_N}}$ is the unique
$\tau_P=\tau_n\otimes \tau_N$-preserving conditional expectation of
$M_n(\mathbb{C})\otimes N$ onto $M_n(\mathbb{C})\otimes {1_N}$\,. Then (\ref{eq2}) follows and the
implication is proved.

Conversely, assume
that $(ii)$ holds. Define maps $\alpha$, $\beta\colon
M_n(\mathbb{C})\rightarrow M_n(\mathbb{C})\otimes N$ by
$\alpha(x)\colon =u^*(x\otimes 1_N)u$\,, respectively, $\beta(x)\colon=x\otimes 1_{N}$\,, for all
$x\in M_n(\mathbb{C})$\,. Then $\alpha$ and $\beta$ are $(\tau_n\,,
\tau_n\otimes \tau_N)$-Markov $*$-monomorphisms of $M_n(\mathbb{C})$
into $M_n(\mathbb{C})\otimes N$ satisfying $T=\beta^*\circ
\alpha$\,, which proves that $T$ is factorizable.

Next we prove the implication $(ii)\Rightarrow (iii)$\,. Assume that
$(ii)$ holds and choose a von Neumann algebra $N$ with a normal faithful tracial state $\tau_N$ and a unitary operator $u\in
M_n(\mathbb{C})\otimes N$ satisfying (\ref{eq2}). Since $a_1\,, \ldots \,, a_d\in
M_n(\mathbb{C})$ are linearly independent, we can extend the set
$\{a_1\,, \ldots\,, a_d\}$ to an algebraic basis $\{a_1\,, \ldots
\,, a_{n^2}\}$ for $M_n(\mathbb{C})$\,. Then
$u$ has a representation of the form $u=
\sum_{i=1}^{n^2} a_i\otimes v_i$\,,
where $v_1\,, \ldots \,, v_{n^2}\in N$\,. By (\ref{eq1}) and
(\ref{eq2}) we deduce that
\begin{eqnarray}\label{eq3}
\sum\limits_{i=1}^d a_i^*xa_i= (\text{id}_{M_n(\mathbb{C})}\otimes
\tau_N)\left(\sum\limits_{i, j=1}^{n^2} a_i^*xa_j\otimes v_i^*
v_j\right)= \sum\limits_{i, j=1}^{n^2} \tau_N(v_i^*v_j)a_i^*xa_j\,.
\end{eqnarray}
For $a\in M_n(\mathbb{C})$ we let $L_a$ and $R_a$
denote, respectively, the operators of left and right multiplication by $a$ on
$M_n(\mathbb{C})$\,, i.e., $L_ax=ax$\,, $R_ax=xa$\,, $x\in M_n(\mathbb{C})$\,.
It is well-known that the map $\sum_{k=1}^r
a_k\otimes b_k \mapsto \sum_{k=1}^r L_{a_k}\otimes R_{b_k}$
defines a vector space isomorphism of $M_n(\mathbb{C})\otimes
M_n(\mathbb{C})$ onto ${\mathcal B}(M_n(\mathbb{C}))$\,. By
(\ref{eq3}),
\begin{eqnarray}\label{eq4}
\sum\limits_{i=1}^d a_i^*\otimes a_i = \sum\limits_{i, j=1}^{n^2}
\tau_N(v_i^* v_j)a_i^*\otimes a_j\,.
\end{eqnarray}
Moreover, the set $\{a_i^*\otimes a_j: 1\leq i, j\leq n\}$ is an
algebraic basis for $M_n(\mathbb{C})\otimes M_n(\mathbb{C})$\,, so
in particular, this set is linearly independent. Therefore
(\ref{eq4}) implies that
\begin{eqnarray*}\tau_N(v_i^*v_j)&=&\left\{\begin{array}{ll}
                              \delta_{ij} & \,\mbox{if} \;\; \;1\leq i,
j\leq d\\
                                       0 & \,\mbox{else}\,.\\
                                \end{array}
                        \right.
\end{eqnarray*}
In
particular, $\tau_N(v_i^* v_i)=0$ for $i> d$\,, which by the
faithfulness of $\tau_N$ implies that $v_i=0$\,, for all $i> d$\,.
Hence $u=\sum_{i=1}^d a_i \otimes v_i$\,, which proves
$(iii)$\,.

It remains to prove that $(iii)$ implies $(ii)$\,.
Choose $(N\,, \tau_N)$ and operators $v_1\,, \ldots \,, v_d\in N$ as
in $(iii)$\,. Then $u\colon = \sum_{i=1}^d
a_i\otimes v_i$ is a unitary operator in $M_n(\mathbb{C})\otimes N$ and
$\tau(v_i^*v_j)=\delta_{ij}$\,, for $1\leq i, j\leq d$\,. Thus
\[ T(x)=\sum\limits_{i=1}^d a_i^*xa_i=\sum\limits_{i, j=1}^n
\tau_N(v_i^*v_j)a_i^*xa_j=(\text{id}_{M_n(\mathbb{C})}\otimes
\tau_N)(u^*(x\otimes 1_N)u)\,, \] which gives $(ii)$ and the proof is
complete.
\end{proof}

\begin{cor}\label{cor1}
Let $T\colon M_n(\mathbb{C})\rightarrow M_n(\mathbb{C})$ be a
$\tau_n$-Markov map of the form (\ref{eq1}),
where $a_1\,, \ldots\,, a_d\in M_n(\mathbb{C})$ and
$\sum_{i=1}^d a_i^* a_i=\sum_{i=1}^d a_i a_i^* =1_n$\,.
If $d\geq 2$ and the set $\{a_i^*a_j: 1\leq i, j\leq d\}$ is
linearly independent, then $T$ is not factorizable.
\end{cor}

\begin{proof}
Assume
that $T$ is factorizable. Since the linear independence of the set $\{a_i^*a_j: 1\leq i,
j\leq d\}$ implies that the set $\{a_i: 1\leq i\leq d\}$ is linearly
independent, as well, Theorem \ref{th1} applies. Hence, by the equivalence
$(ii)\Leftrightarrow (iii)$ therein, there
exists a finite von Neumann algebra $N$ with a normal, faithful
tracial state $\tau_N$ and operators $v_1\,, \ldots\,, v_d\in N$
such that $u\colon =\sum_{i=1}^d a_i\otimes
v_i\in M_n(\mathbb{C})\otimes N=M_n(N)$ is a unitary operator and
$\tau_N(a_i^*a_j)=\delta_{ij}$\,, for $1\leq i, j\leq d$\,. Then
\[ \sum\limits_{i, j=1}^d a_i^*a_j\otimes
(v_i^*v_j-\delta_{ij}1_N)=u^*u-\left(\sum\limits_{i=1}^d
a_i^*a_i\right)\otimes
1_N=1_{M_n(N)}-{1_n}\otimes {1_N}=0_{M_n(N)}\,. \] By
the linear independence of the set $\{a_i^*a_j: 1\leq i, j\leq d\}$
it follows that for every functional $\phi\in N^*$\,, we have $\phi(v_i^*v_j-\delta_{ij}1_N)=0$\,, for all $1\leq i, j\leq d$\,.
Hence
\[ v_i^*v_j=\delta_{ij}1_N\,, \quad 1\leq i, j\leq d\,. \]
Since $d\geq 2$\,, we infer in particular that $v_1^*v_2=0_N$ and
$v_1^*v_1=v_2^*v_2=1_N$\,. The latter condition ensures that $v_1$
and $v_2$ are unitary operators, since $N$ is a finite von Neumann algebra. But this leads to a contradiction with the fact that
$v_1^*v_2=0_N$\,. This proves that $T$ is not factorizable.
\end{proof}

Let $\text{Aut}(M)$ denote the set of $*$-automorphisms of a von
Neumann algebra $M$\,. If $M=M_n(\mathbb{C})$\,, for some $n\in \mathbb{N}$\,, then $\text{Aut}(M_n(\mathbb{C}))=\{\text{ad}(u): u\in
{\mathcal U}(n)\}$\,, where $\text{ad}(u)x=uxu^*$, $x\in
M_n(\mathbb{C})$\,, and ${\mathcal U}(n)$ denotes the unitary group
of $M_n(\mathbb{C})$\,. Since ${\mathcal B}(M_n(\mathbb{C}))$ is
finite dimensional, the convex hull of
$\text{Aut}(M_n(\mathbb{C}))$\,, denoted by
$\text{conv}(\text{Aut}(M_n(\mathbb{C}))$\,, is closed in the norm topology on ${\mathcal B}(M_n(\mathbb{C}))$\,. Further, let us denote by
${\mathcal F}{\mathcal M}(M_n(\mathbb{C}))$ the set of factorizable $\tau_n$-Markov maps on
$M_n(\mathbb{C})$\,. By Remark \ref{rem787877777} $(b)$\,,
\begin{equation}\label{eq400068}
\text{conv}(\text{Aut}(M_n(\mathbb{C}))\subseteq {\mathcal F}{\mathcal M}(M_n(\mathbb{C}))\,.
\end{equation}
Note that for $n=2$ the two sets above are equal, as they are further equal to the set of $\tau_2$-Markov maps on $M_2(\mathbb{C})$\,, as shown by K\"ummerer in \cite{Ku2}.

\begin{prop}\label{prop2}
Let $T: M_n(\mathbb{C})\rightarrow M_n(\mathbb{C})$ be a
$\tau_n$-Markov map written in the form
(\ref{eq1})\,, where $a_1\,, \ldots \,, a_d\in M_n(\mathbb{C})$ are
linearly independent and $\sum_{i=1}^d a_i^*
a_i=\sum_{i=1}^d a_i a_i^* =1_n$\,. Then the following
conditions are equivalent:
\begin{enumerate}
\item [$(1)$] $T\in \text{conv}(\text{Aut}(M_n(\mathbb{C}))$\,.
\item [$(2)$] $T$ satisfies condition $(ii)$ of Theorem \ref{th1} with $N$
abelian.
\item [$(3)$] $T$ satisfies condition $(iii)$ of Theorem \ref{th1} with $N$
abelian.
\end{enumerate}
\end{prop}

\begin{proof}
We first show that $(1)\Rightarrow (2)$\,. Assume that $T\in
\text{conv}(\text{Aut}(M_n(\mathbb{C}))$\,. Then there exist $u_1\,, \ldots \,, u_s\in {\mathcal U}(n)$ and positive real numbers $c_1\,, \ldots \,, c_s$ with sum equal to 1, for some positive integer $s$, so that
\begin{eqnarray*}
Tx= \sum\limits_{i=1}^s c_i u_i^*x u_i\,, \quad x\in
M_n(\mathbb{C})\,.
\end{eqnarray*}
 Next, consider the
abelian von Neumann algebra $N \colon =l^\infty (\{1\,, \ldots\,, s\})$
with faithful tracial state $\tau_N$ given by $\tau_N(a)\colon =
\sum_{i=1}^s c_i a_i$\,, $a=(a_1\,, \ldots \,, a_s)\in N$\,. Set $u\colon=(u_1\,, \ldots \,, u_s)\in l^\infty
(\{1\,, \ldots \,, s\}\,, M_n(\mathbb{C}))=M_n(\mathbb{C})\otimes
N$\,. Then $u$ is unitary and relation (\ref{eq2}) is
satisfied.

We now show that $(2)\Rightarrow (1)$\,. Assume that
$(2)$ holds, i.e., there exists an abelian von Neumann algebra $N$
with a normal faithful tracial state $\tau_N$ and a unitary operator
$u\in M_n(N)$ such that (\ref{eq2}) is satisfied. Let
$\widehat{N}$ denote the spectrum of $N$ (i.e., the set of
non-trivial multiplicative linear functionals on $N$)\,. Then
$\widehat{N}$ is compact in the w$^*$-topology, $N\simeq
C(\widehat{N})$ and $\tau_N$ corresponds to a regular Borel
probability measure $\mu$ on $\widehat{N}$\,. By identifying $N$
with $C(\widehat{N})$\,, we have $u\in
M_n(C(\widehat{N}))=C(\widehat{N}\,, M_n(\mathbb{C}))$ and
\begin{eqnarray*}
Tx= \int_{\widehat{N}} {u(t)}^* x u(t) d\mu(t)\,, \quad
x\in M_n(\mathbb{C})\,.
\end{eqnarray*}
Thus $T$ lies in the norm-closure of ${\text{conv}(\text{ad}(u(t)^*) : t\in
\widehat{N})}.$ Since
$\text{conv}(\text{Aut}(M_n(\mathbb{C}))$ is a closed set in
${\mathcal B}(M_n(\mathbb{C}))$\,, condition $(1)$ follows.

The implication $(2)\Rightarrow (3)$ follows
immediately from the proof of the corresponding implication
$(ii)\Rightarrow (iii)$ in Theorem \ref{th1}\,.
\end{proof}

\begin{cor}\label{cor2}
Let $T\colon M_n(\mathbb{C})\rightarrow M_n(\mathbb{C})$ be a
$\tau_n$-Markov map of the form (\ref{eq1}),
where $a_1\,, \ldots\,, a_d\in M_n(\mathbb{C})$ are self-adjoint,
$\sum_{i=1}^d a_i^2 =1_n$ and $a_ia_j=a_ja_i$\,, for all $1\leq i, j\leq d$\,.
Then the following hold:
\begin{itemize}
\item [$(a)$] $T$ is factorizable.
\item [$(b)$] If $d\geq 3$ and the set $\{a_ia_j: 1\leq i\leq j\leq
d\}$  is linearly independent, then $T\notin
\text{conv}(\text{Aut}(M))$\,.
\end{itemize}
\end{cor}

\begin{proof}
The proof of $(a)$ is inspired by the proof of Theorem 3 in
\cite{Ric}\,. Let $N$ be the CAR-algebra over a
$d$-dimensional Hilbert space $H$ with orthonormal basis $e_1\,,
\ldots \,, e_d$\,, and let $a(e_i)$\,, $1\leq i\leq d$ be the
corresponding annihilation operators. Then $N\simeq
M_{2^d}(\mathbb{C})$ and the operators defined by
\[ v_i\colon =a(e_i)+a(e_i)^*\,, \quad 1\leq i\leq d \]
form a set of anti-commuting self-adjoint unitaries (see, e.g., \cite{BR}). Set now $u\colon
=\sum_{i=1}^d a_i\otimes v_i\in M_n(\mathbb{C})\otimes N$\,.
Then $u$ is unitary since
\begin{eqnarray*}
u^*u\,=\,\sum\limits_{i,
j=1}^d a_ia_j\otimes v_iv_j\,=\,\frac{1}{2}\sum\limits_{i, j=1}^d
(a_ia_j+a_ja_i)\otimes v_iv_j&=& \frac{1}{2}\sum\limits_{i, j=1}^d
a_ia_j\otimes (v_iv_j+v_jv_i)\\&=&\sum\limits_{i=1}^d a_i^2\otimes
1_N=1_{M_n(N)}\,. \end{eqnarray*}
Moreover,
$\tau_N(v_i^*v_j)=\tau_N((v_iv_j+v_jv_i)/2)=\delta_{ij}$\,,
$1\leq i, j\leq d$\,. Hence, by the implication $(iii)\Rightarrow
(i)$ of Theorem \ref{th1}, we deduce that $T$ is factorizable.

We now prove $(b)$\,. Assume that $d\geq 3$ and that $\{a_ia_j:
1\leq i\leq j\leq d\}$ is linearly independent. In particular, the
set $\{a_i: 1\leq i\leq d\}$ is linearly independent. If $T\in
\text{conv}(\text{Aut}(M_n(\mathbb{C}))$\,, then by Proposition
\ref{prop2}, there exists an abelian von Neumann algebra $N$ with
normal faithful tracial state $\tau_N$ and operators $v_1\,,
\ldots\,, v_d\in N$ such that the operator
$u\colon=\sum_{i=1}^d a_i\otimes v_i\in M_n(\mathbb{C})\otimes N=M_n(N)$ is unitary.
Therefore, $1_{M_n(N)}=u^*u=\sum_{i, j=1}^d a_i a_j\otimes v_i^* v_j$\,.
Using the fact that $a_ia_j=a_ja_i$\,, for all $1\leq i, j\leq d$\,, and that $\sum_{i=1}^d a_i^2=1_n$\,, we infer that
\[ \sum_{i=1}^d a_i^2 \otimes (v_i^*v_i-1_N)+\sum_{1\leq i< j\leq d}
a_ia_j\otimes (v_i^*v_j+v_j^*v_i)=0_{M_n(N)}\,. \] By the linear independence
of the set $\{a_ia_j: 1\leq i\leq j\leq d\}$\,, it follows that
$v_i^*v_i=1_N$\,, for $1\leq i\leq d$\,, and, respectively, that
$v_i^*v_j+v_j^*v_i=0_N$\,, for $1\leq i< j\leq d$\,. Since $N\simeq C(\widehat{N})$\,, we deduce that
\[ |v_i(t)|=1\,, \quad t\in \widehat{N}\,, \quad 1\leq i\leq d\,, \]
and, respectively,
\[ \text{Re}(\overline{v_i(t)}v_j(t))=0\,, \quad t\in \widehat{N}\,, \quad 1\leq i\neq j\leq d\,, \]
Since $d\geq 3$\,, it follows that
$\overline{v_1(t)}v_2(t)$\,, $\overline{v_2(t)}v_3(t)$ and
$\overline{v_3(t)}v_1(t)$ are purely imaginary complex numbers, for
all $t\in \widehat{N}$\,. Hence the product of these numbers is
also purely imaginary. On the other hand, this product equals
$|v_1(t)|^2|v_2(t)|^2|v_3(t)|^2=1$\,, which gives rise to a contradiction. We conclude that $T\notin
\text{conv}(\text{Aut}(M))$\,.
\end{proof}

We now discuss the case of Schur multipliers.
The following fact is probably well-known, but we include a proof
for completeness.

\begin{prop}\label{prop3}
Let $B=(b_{ij})_{i, j=1}^n \in M_n(\mathbb{C})$ and let $T_B\colon
M_n(\mathbb{C})\rightarrow M_n(\mathbb{C})$ be its corresponding
Schur multiplier, i.e.,
$T_B(x)=(b_{ij}x_{ij})_{i, j=1}^n$\,, for all $x=(x_{ij})_{i, j=1}^n\in
M_n(\mathbb{C})$\,.
The following conditions are equivalent:
\begin{enumerate}
\item $T_B$ is positive
\item $T_B$ is completely positive
\item There exist diagonal matrices $a_1\,, \ldots \,, a_d\in
M_n(\mathbb{C})$ such that
\begin{eqnarray}\label{eq9}
T_B (x)=\sum\limits_{i=1}^d a_i^*xa_i\,, \quad x\in
M_n(\mathbb{C})\,.
\end{eqnarray}
\item There exist linearly independent diagonal matrices $a_1\,, \ldots \,, a_d\in
M_n(\mathbb{C})$ such that (\ref{eq9}) holds. \item $B$ is a
positive semi-definite matrix, i.e., $B=B^*$ and all eigenvalues of
$B$ are non-negative.
\end{enumerate}
\end{prop}

\begin{proof}
The series of implications $(4)\Rightarrow (3)\Rightarrow
(2)\Rightarrow (1)$ is trivial, so we only have to prove that
$(1)\Rightarrow (5)\Rightarrow (4)$\,. Assume that $T_B$ is
positive, i.e., $T_B((M_n(\mathbb{C}))_+)=(M_n(\mathbb{C}))_+$\,,
where $(M_n(\mathbb{C}))_+$ denotes the set of positive
semi-definite $n\times n$ complex matrices. Clearly, the matrix
$x_0$ whose entries are all equal to $1$ belongs to
$(M_n(\mathbb{C}))_+$\,, and therefore
$B=T_B(x_0)\in (M_n(\mathbb{C}))_+$\,. This shows that $(1)\Rightarrow (5)$\,.

To prove $(5)\Rightarrow (4)$\,, assume that $B=B^*$ with
non-negative eigenvalues. Then $B=C^*DC$\,, where $C$ is unitary and
$D=\text{diag}(\lambda_1\,, \ldots\,, \lambda_n)$ is a diagonal
matrix whose diagonal entries are the eigenvalues of $B$ repeated
according to multiplicity. In particular, $\lambda_i\geq 0$\,, for
$1\leq i\leq n$\,. Let $d\colon=\text{rank}(B)=\text{rank}(D)$\,. We
may assume that $\lambda_1\,, \ldots \,, \lambda_d> 0$ and
$\lambda_{d+1}=\ldots =\lambda_n=0$\,. For any $1\leq i\leq d$
consider now the diagonal $n\times n$  matrix given by
$a_i\colon =\sqrt{\lambda_i}\text{diag}(c_{i1}\,, \ldots c_{in})$\,,
where $(c_{i1}\,, \ldots \,, c_{in})$ is the $i$-th row of $C$\,.
Then $a_1\,, \ldots \,, a_d$ are linearly independent and one checks
easily that (\ref{eq9}) holds.
\end{proof}

Let $B\in M_n(\mathbb{C})$\,. By Proposition \ref{prop3}\,, the
Schur multiplier $T_B$ associated to the matrix $B$ is a $\tau_n$-Markov map if and only if $B$ is positive semi-definite and
$b_{11}=b_{22}=\ldots =b_{nn}=1$\,, because the latter
condition is equivalent to having
$T_B(1_n)=1_n$ and $\tau_n\circ T_B=\tau_n$\,.

\begin{rem}\label{rem100}\rm
In \cite{Ric} E. Ricard proved that if $B=(b_{ij})_{i, j=1}^n\in
M_n(\mathbb{R})$ is a positive semi-definite matrix whose diagonal entries are all equal to 1, then the associated Schur multiplier $T_B$ is always factorizable.
This result can also be obtained from Corollary \ref{cor2} $(a)$. Indeed, under the above hypotheses,
$B=C^t DC$\,, where $C$ is an orthogonal matrix and
$D=\text{diag}\{d_1\,, \ldots\,, d_n\}$ is a diagonal matrix with
$\lambda_i\geq 0$\,. Let $d:=\text{rank}(D)$\,. Then, following the
proof of the implication $(5)\Rightarrow (4)$ in Proposition
\ref{prop3}, we deduce that
$T_B(x)=\sum_{i=1}^d a_i^*xa_i$\,, for all $x\in
M_n(\mathbb{C})$\,, where $a_1\,, \ldots\,, a_d$ are linearly
independent diagonal matrices with $a_i=a_i^*$\,, $1\leq i\leq d$
(since the entries of $C$ are real numbers). Moreover,
$\sum_{i=1}^d
a_i^2=\sum_{i=1}^d a_i^*a_i=T_B(1_n)=1_n$\,, and
$a_ia_j=a_ja_i$\,, for $1\leq i, j\leq d$\,. It then follows from
Corollary \ref{cor2} $(a)$ that $T_B$ is factorizable.
\end{rem}

We end this section with a general characterization of factorizable Schur multipliers, which turns out to be useful for applications.

\begin{lemma}\label{lem222222}
Let $B=(b_{ij})_{i, j=1}^n$ be a positive semi-definite $n\times n$ complex matrix having all diagonal entries equal to 1. Then the associated
Schur multiplier $T_B$ is a factorizable $\tau_n$-Markov map if and only if there exists a
finite von Neumann algebra $N$ with normal faithful tracial state
$\tau_N$ and unitaries $u_1\,, \ldots \,, u_n\in N$ such that
\begin{eqnarray}\label{eq99909}
b_{ij}=\tau_N(u_i^*u_j)\,, \quad 1\leq i, j\leq n\,.
\end{eqnarray}
\end{lemma}

\begin{proof}
Assume that $T_B$ is factorizable. Then by Theorem \ref{th1}\,,
there exists a finite von Neumann algebra $N$ with
normal faithful tracial state $\tau_N$\,, and a unitary $u\in
M_n(N)=M_n(\mathbb{C})\otimes N$ such that
\begin{eqnarray}\label{eq4444454} T_B(x)=(\text{id}_{M_n(\mathbb{C})}\otimes \tau_N)(u^*(x\otimes
1_N)u)\, , \quad x\in M_n(\mathbb{C})\,. \end{eqnarray} It follows that
\begin{eqnarray}\label{eq44444542} \tau_n(yT_B(x))=(\tau_n\otimes
\tau_N)((y\otimes 1_N)u^*(x\otimes 1_N)u)\,, \quad x, y\in
M_n(\mathbb{C})\,.
\end{eqnarray}
Let $(e_{jk})_{1\leq j, k\leq n}$ be the standard matrix units in
$M_n(\mathbb{C})$\,. Then $u=\sum_{i,
k=1}^n e_{jk}\otimes u_{jk}$\,, where $u_{jk}\in N$\,, $1\leq j,
k\leq n$\,, and $u^*=\sum_{i, k=1}^n e_{kj}\otimes u_{jk}^*$\,.
Consider now $j\,, k\in \{1\,, \ldots\,, n\}$\,, $j\neq k$\,. By applying
(\ref{eq44444542}) to $x=e_{jj}$ and $y=e_{kk}$\,, we get
$\tau_n(e_{kk} T_B(e_{jj}))=b_{jj}\tau_n(e_{kk}e_{jj})=0$\,. Therefore,
\begin{eqnarray*} 0=(\tau_n\otimes \tau_N)((e_{kk}\otimes
1_N)u^*(e_{jj}\otimes 1_N)u)&=&(\tau_n\otimes \tau_N)((e_{kk}\otimes
1_N)u^*(e_{jj}\otimes 1_N)u (e_{kk}\otimes 1_N))\\&=&(\tau_n\otimes
\tau_N)(e_{kk}\otimes u_{jk}^*u_{jk})\\&=&(1/n)
\tau_N(u_{jk}^*u_{jk})\,.
\end{eqnarray*}
By the faithfulness of $\tau_N$\,, $u_{jk}=0_N$ for $j\neq k$\,. Thus $u=\sum_{j=1}^n e_{jj}\otimes u_{jj}$\,.
For $1\leq j, k\leq n$ we then get
\begin{eqnarray*}
b_{jk}=b_{jk}n\tau_n(e_{kj}e_{jk})=n\tau_n(e_{kj}T_B(e_{jk}))
&=& n(\tau_n\otimes \tau_N)((e_{kj}\otimes 1_N)u^*(e_{jk}\otimes
1_N)u)\\&=&n(\tau_n\otimes \tau_N)(e_{kk}\otimes u_{jj}^*u_{kk})\\&=&
\tau_N(u_{jj}^*u_{kk})\,.
\end{eqnarray*}
Hence (\ref{eq99909}) holds with $u_j=u_{jj}$\,, for $1\leq j\leq
n$\,.

Conversely, if (\ref{eq99909}) holds for a set of $n$
unitaries $u_1\,, \ldots \,, u_n$ in a finite von Neumann algebra
$N$ with normal, faithful, tracial state $\tau_N$\,, then the operator
$u\colon =\sum_{j=1}^n e_{jj}\otimes u_j$
is a unitary in $M_n(\mathbb{C})\otimes N$ and one checks easily
that (\ref{eq4444454}) holds. Hence, by Theorem \ref{th1}, $T_B$ is factorizable.
\end{proof}

\section{Examples}\setcounter{equation}{0}
We begin by exhibiting an example of a non-factorizable $\tau_3$-Markov maps on $M_3(\mathbb{C})$\,.

\begin{exam}\label{exp1}\rm
Set \[ a_1={\frac1{\sqrt{2}}}\left(
\begin{array}
[c]{ccc}%
0 & 0 & 0\\
0 & 0 & -1\\
0 & 1 & 0
\end{array}
\right)\,, \quad a_2={\frac1{\sqrt{2}}}\left(
\begin{array}
[c]{ccc}%
0 & 0 & 1\\
0 & 0 & 0\\
-1 & 0 & 0
\end{array}
\right)\,, \quad a_3={\frac1{\sqrt{2}}}\left(
\begin{array}
[c]{ccc}%
0 & -1 & 0\\
1 & 0 & 0\\
0 & 0 & 0
\end{array}
\right)\,. \] Then $\sum_{i=1}^3 a_i^*a_i=\sum_{i=1}^3
a_ia_i^*=1_3$, and hence the operator $T$ defined by
$Tx\colon= \sum\limits_{i=1}^3 a_i^* x a_i$\,, for all $x\in
M_n(\mathbb{C})$
is a $\tau_3$-Markov map. If $T$ were
factorizable, then by the implication $(i)\Rightarrow (iii)$ in
Theorem \ref{th1}\,, there would exist a finite von Neumann algebra $N$
with a normal faithful tracial state $\tau_N$ and elements $v_1\,,
v_2\,, v_3\in N$ such that the operator
\begin{eqnarray*}
u\colon = \sum\limits_{i=1}^3 a_i\otimes
v_i={\frac1{\sqrt{2}}}\left(
\begin{array}
[c]{ccc}%
0 & -v_3 & v_2\\
v_3 & 0 & -v_1\\
-v_2 & v_1 & 0
\end{array}
\right)
\end{eqnarray*}
is unitary, but as observed in \cite{HI} (pp. 282-283), this is
impossible. Indeed, since $u^*u=1_N$\,, we have
\begin{eqnarray}\label{eq7}
v_1^*v_1+v_2^*v_2=v_2^*v_2+v_3^*v_3=v_3^*v_3+v_1^*v_1=2\,{1_N}
\end{eqnarray}
and, respectively,
\begin{eqnarray}\label{eq8}
v_1^*v_2=v_2^*v_3=v_3^*v_1=0_N\,.
\end{eqnarray}
Note that (\ref{eq7}) implies that
$v_1^*v_1=v_2^*v_2=v_3^*v_3=1_N$\,, and since $N$ is finite,
it follows that $v_1$\,, $v_2$ and $v_3$ are unitary operators, which contradicts
(\ref{eq8})\,. This shows that $T$ is not factorizable.

Alternatively, one can check that $\{a_i^*a_j: 1\leq i, j\leq 3\}$
is a linearly independent set in $M_3(\mathbb{C})$ and then use
Corollary \ref{cor1} to prove that $T$ is not factorizable.
\end{exam}

We now present some concrete examples of non-factorizable Schur multipliers.

\begin{exam}\label{exp2}\rm
Following an example constructed in \cite{CV}, for $0\leq s\leq 1$ set
\begin{eqnarray*}
B(s)\colon=\left(
\begin{array}
[c]{cccc}%
1 & \sqrt{s} & \sqrt{s} & \sqrt{s}\\
\sqrt{s} & s & s & s\\
\sqrt{s} & s & s & s\\
\sqrt{s} & s & s & s
\end{array}
\right) + (1-s){\left(
\begin{array}
[c]{cccc}%
0 & 0 & 0 & 0\\
0 & 1 & \omega & \overline{\omega}\\
0 & \overline{\omega} & 1 & \omega\\
0 & \omega & \overline{\omega} & 1
\end{array}
\right)}\,,
\end{eqnarray*}
where $\omega\colon =e^{{{i2}{\pi}}/3}={-1}/2+i{{\sqrt{3}}/2}$ and
$\overline{\omega}$ is the
complex conjugate of $\omega$\,.
Note that $B(s)$ is positive semi-definite, since
\begin{eqnarray*}
B(s)=x_1(s)^*x_1(s)+x_2(s)^*x_2(s)\,,
\end{eqnarray*}
where $x_1(s)=(1\,, \sqrt{s}\,, \sqrt{s}\,, \sqrt{s})$ and
$x_2(s)={\sqrt{1-s}}(0\,, 1\,, \omega\,, \overline{\omega})$\,.
Moreover, $b_{11}=b_{22}=b_{33} =b_{44}=1$\,. Thus  $T_{B(s)}$
is a $\tau_4$-Markov map for all $0\leq s\leq
1$\,.

We claim that for $0<s<1$\,, the map $T_{B(s)}$ is not
factorizable. To prove this, we will use Corollary \ref{cor1}\,. Let $0< s< 1$ and
observe that
\begin{eqnarray*}
T_{B(s)}(x)=a_1(s)^*xa_1(s)+a_2(s)^*xa_2(s)\,, \quad x\in M_4(\mathbb{C})\,,
\end{eqnarray*}
where $a_1(s)$ and $a_2(s)$ are the diagonal $4\times 4$ matrices
\[ a_1(s)\colon =\text{diag}(1\,, \sqrt{s}\,, \sqrt{s}\,, \sqrt{s})\,,
\quad  a_2(s)\colon =\sqrt{1-s}\,\text{diag}(0\,, 1\,, \omega\,,
\overline{\omega})\,. \] It is elementary to check that the
following four matrices: $a_1(s)^*a_1(s)=\text{diag}(1\,, s\,, s\,,
s)$\,, $a_2(s)^*a_2(s)=(1-s)\,\text{diag}(0\,, 1\,, 1\,, 1)$\,,
$a_1(s)^*a_2(s)=\sqrt{s}\,\text{diag}(0\,, 1\,, \omega\,,
\overline{\omega})$ and $a_2(s)^*a_1(s)=\sqrt{s}\,\text{diag}(0\,, 1\,,
\overline{\omega}\,, \omega)$ are linearly independent. Hence, by the above-mentioned corollary,
$T_{B(s)}$ is not factorizable.

Note that for $s=1/3$\,, $B(s)$ has
a particularly simple form, namely, \begin{eqnarray*} B(1/3)=\left(
\begin{array}
[c]{cccc}%
1 & 1/{\sqrt{3}} & 1/{\sqrt{3}} & 1/{\sqrt{3}}\\
1/{\sqrt{3}} & 1 & {i}/{\sqrt{3}} & -{{i}/{\sqrt{3}}}\\
1/{\sqrt{3}} & -{{i}/{\sqrt{3}}} & 1 & {i}/{\sqrt{3}}\\
1/{\sqrt{3}} & {i}/{\sqrt{3}} & -{{i}/{\sqrt{3}}} & 1
\end{array}
\right)
\end{eqnarray*}
The above $4\times 4$ matrix examples can easily be generalized to examples of non-factorizable Schur multipliers on $(M_n(\mathbb{C}), \tau_n)$\,, for all $n\geq 4$\,, by setting
\begin{equation*}
B(s)\colon = x_1(s)^* x_1(s)+x_2(s)^* x_2(s)\,,
\end{equation*}
where $0< s< 1$ and $x_1\,, x_2$ are the column vectors in ${\mathbb{C}}^n$ given by
\[ x_1(s)\colon = (1\,, \sqrt{s}\,, \sqrt{s}\,, \ldots\,, \sqrt{s})\,, \quad x_2(s)\colon = {\sqrt{1-s}}(0\,, 1\,, \rho\,, \rho^2\,, \ldots\,, \rho^{n-2})\,, \]
where $\rho=e^{{i\,2 \pi}/(n-1)}$\,. Further, let $a_1(s)$ and $a_2(s)$ be the corresponding $n\times n$ diagonal matrices. Then the linear independence of the set $\{(a_j(s))^* a_k(s): 1\leq j, k\leq 2\}$ follows from the computation
\[ \text{det}\, \left(
\begin{array}
[c]{ccc}%
1 & 1 & 1\\
1 & \rho & \rho^2\\
1 & \bar{\rho} & {\bar{\rho}}^{\,2}\\
\end{array}
\right) = {\bar{\rho}}^{\,2}(\rho +1)(\rho -1)^3\neq 0\,, \]
where $\bar{\rho}$ is the conjugate of $\rho$\,. Then, an application of Corollary \ref{cor1} shows that the Schur multiplier $T_{B(s)}$ is not factorizable.

\end{exam}

\begin{exam}\label{exp3}\rm
Let $\beta=1/{\sqrt{5}}$ and set \[ B\colon= \left(
\begin{array}
[c]{cccccc}%
1 & \beta & \beta & \beta & \beta & \beta\\
\beta & 1 & \beta & -\beta & -\beta & -\beta\\
\beta & \beta & 1 & \beta & -\beta & -\beta\\
\beta & -\beta & \beta & 1 & \beta & -\beta\\
\beta & -\beta & -\beta & \beta & 1 & \beta\\
\beta & \beta & -\beta & -\beta & \beta & 1\\
\end{array}
\right)\,. \] We claim that $T_B$ is a
factorizable $\tau_6$-Markov map on $M_6(\mathbb{C})$\,, but
$T_B\notin \text{conv}(\text{Aut}(M_6(\mathbb{C})))$\,. To prove
this, observe first that since $\cos ({2\pi}/{5})=\cos
({8\pi}/{5})=(-1+\sqrt{5})/4$ and $\cos ({4\pi}/{5})=\cos
({6\pi}/{5})=(-1-\sqrt{5})/4$\,, then $B$ can be written in the form
\[ B=x^*_1x_1+x^*_2x_2+x^*_3x_3\,, \] where $x_1\colon =(1, 1/{\sqrt{5}},
1/{\sqrt{5}}\,, 1/{\sqrt{5}}\,, 1/{\sqrt{5}}\,, 1/{\sqrt{5}})$\,,
$x_2\colon =\sqrt{2/5}(0, 1, e^{i2\pi/5}\,, e^{i4\pi/5}\,, e^{i6\pi/5}\,,
e^{i8\pi/5})$ and $x_3\colon =\overline{x_2}=\sqrt{2/5}(0, 1,
e^{-i2\pi/5}\,, e^{-i4\pi/5}\,, e^{-i6\pi/5}\,, e^{-i8\pi/5})$\,.
Hence $B$ is positive semi-definite. By Remark \ref{rem100}\,, $T_B$ is
a factorizable $\tau_6$-Markov map on $M_6(\mathbb{C})$\,. Moreover,
\[ T_{B}(x)=\sum\limits_{i=1}^3 b_i^*xb_i\,, \quad x\in
M_6(\mathbb{C})\,, \] where $b_1\colon =\text{diag} \,(1, 1/{\sqrt{5}},
1/{\sqrt{5}}\,, 1/{\sqrt{5}}\,, 1/{\sqrt{5}}\,, 1/{\sqrt{5}})$\,,
$b_2\colon =\sqrt{2/5}\,\text{diag} \,(0, 1, e^{i2\pi/5}\,, e^{i4\pi/5}\,,
e^{i6\pi/5}\,, e^{i8\pi/5})$ and $b_3\colon ={b_2}^*$\,.
Set now \[ a_1\colon =b_1\,, \quad a_2\colon =\frac1{\sqrt{2}}(b_2+b_3)\,,
\quad a_3\colon =\frac1{i\sqrt{2}}(b_2-b_3)\,. \] Then
$T_{B}(x)=\sum_{i=1}^3 a_i x a_i$\,, for all $x\in
M_6(\mathbb{C})$\,. Note that $a_1\,, a_2\,, a_3$ are commuting
self-adjoint (diagonal) matrices with $\sum_{i=1}^3 a_i^2=1_6$\,. Thus, if we knew that
the set
$\{a_ia_j: 1\leq i\leq j\leq 3\}$
is linearly independent, then, by Corollary \ref{cor2} $(b)$ we could conclude that
$T\notin \text{conv}(\text{Aut}(M_n(\mathbb{C})))$\,. Note that the linear
independence of the above set is equivalent to the
linear independence of the set
\begin{eqnarray}\label {eq567658}
\{b_ib_j: 1\leq i\leq j\leq 3\}\,.
\end{eqnarray}
Set $\gamma\colon =e^{{i2\pi}/{5}}$\,. Then the following relations hold:
\begin{eqnarray}\label{eq5121} \,\,\,\,&&b_1^2= \frac{1}{5}\, \text{diag}\,\left(5, 1, 1, 1, 1,
1\right)\,, b_1b_2=\frac{\sqrt{2}}{5} \text{diag}\,(0, 1, \gamma,
\gamma^2, \gamma^3, \gamma^4)\,, b_1b_3=\frac{\sqrt{2}}{5}\,
\text{diag}\,(0, 1, \gamma^4, \gamma^8, \gamma^{12},
\gamma^{16})\\
\,\,\,\,&&b_2^2=\frac{2}{5}\, \text{diag}\,\left(0, 1,
\gamma^2, \gamma^4, \gamma^6, \gamma^8\right)\,,
b_2b_3=\frac{2}{5}\, \text{diag}\,(0, 1, 1,
1, 1, 1)\,, b_3^2=\frac{2}{5} \,\text{diag}\,\left(0, 1, \gamma^3,
\gamma^6, \gamma^9, \gamma^{12}\right)\,.\nonumber \end{eqnarray}
Now let \[ H\colon =\left(\exp\left(i \frac{2\pi
kl}{5}\right)\right)_{0\leq k, l\leq 4}=\left(
\begin{array}
[c]{ccccc}%
1 & 1 & 1 & 1 & 1\\
1 & \gamma & \gamma^2 & \gamma^3 & \gamma^4\\
1 & \gamma^2 & \gamma^4 & \gamma^6 & \gamma^8\\
1 & \gamma^3 & \gamma^6 & \gamma^9 & \gamma^{12}\\
1 & \gamma^4 & \gamma^8 & \gamma^{12} & \gamma^{16}
\end{array}
\right)\,. \] Then $H$ is a complex $5\times 5$ Hadamard matrix, i.e.,
$|h_{i,j}|^2=1$\,, for all $0\leq i, j\leq 4$ and $H^*H=HH^*=5 \,1_5$\,. It follows that the
rows of $H$ are linearly independent. This fact, combined with the relations (\ref{eq5121}), shows that the set in
(\ref{eq567658}) is linearly independent. Hence the assertion is
proved.
\end{exam}

\begin{theorem}\label{th990}
Let $L=(L_{jk})_{1\leq j, k\leq 4}$ be the $4\times 4$ complex matrix
given by
\begin{eqnarray}\label{eq221110}
L\colon =\left(
\begin{array}
[c]{cccc}%
0 & 1/2 & 1/2 & 1/2\\
1/2 & 0 & 1-\omega & 1-\overline{\omega}\\
1/2 & 1-\overline{\omega} & 0 & 1-\omega\\
1/2 & 1-\omega &
1-\overline{\omega} & 0
\end{array}
\right)
\end{eqnarray}
where $\omega\colon =e^{{i 2 \pi}/3}$ and
$\overline{\omega}$ is the complex conjugate
of $\omega$\,. Let $(C(t))_{t\geq 0}$ denote the one-parameter
family of $4\times 4$ complex matrices \begin{eqnarray}\label{eq776765}
C(t)\colon =(e^{-tL_{jk}})_{1\leq j, k\leq 4}\,, \quad t\geq 0\,.
\end{eqnarray}
Then the corresponding Schur multipliers
\begin{eqnarray}\label{eq7767658}
T_t\colon =T_{C(t)}\,, \quad t\geq 0
\end{eqnarray}
form a continuous one-parameter semigroup of $\tau_4$-Markov maps on $M_4(\mathbb{C})$ starting at
$T(0)=\text{id}_{M_4(\mathbb{C})}$\,. Moreover, there exists $t_0>
0$ such that $T(t)$ is not factorizable for any $0< t< t_0$\,.
\end{theorem}

\begin{proof}
Let $((B(s))_{1\leq s\leq 1}$ be the positive semi-definite $4\times
4$ complex matrices considered in Example \ref{exp2}. In particular, the
matrix $B(1)$ has all entries equal to 1\,.
Note that the matrix $L$ given by (\ref{eq221110}) is the first
derivative of $B(s)$ at $s=1$\,, i.e.,
\[ L={\frac{dB(s)}{ds}}_{\vert_{s=1}}=\lim\limits_{s\nearrow 1}
\frac {B(s)-B(1)}{s-1}\,. \] Since $B(s)$ is positive semi-definite
for all $0\leq s\leq 1$\,, we have that
\[ \sum\limits_{j, k=1}^4 (B(s)-B(1))c_j\bar{c_k}\geq 0\,, \]
whenever $c_1\,, \ldots\,, c_4\in \mathbb{C}$ and $c_1+\,, \ldots
+c_4=0$\,. This implies that $\sum_{j, k=1}^4 L_{jk} c_j\bar{c_k}\geq 0$\,,
i.e., $L$ is a conditionally negative definite matrix. By
Schoenberg's theorem (see, e.g., \cite{BCR}), the matrices
$C(t)$\,, $t\geq 0$ given by (\ref{eq776765}) are all positive
semi-definite. Moreover, since $L_{11}=L_{22}=L_{33}=L_{44}=0$\,, we
also have $C(t)_{11}=C(t)_{22}=C(t)_{33}=C(t)_{44}=0$\,, for all
$t\geq 0$\,. Hence the Schur multipliers $T(t)=T_{C(t)}$\,, $t\geq
0$ are all  $\tau_4$-Markov maps. Clearly, the family
$(T(t))_{t\geq 0}$ forms a continuous one-parameter semigroup of
Schur multipliers starting at $T(0)=\text{id}_{M_4(\mathbb{C})}$\,.
Now set
\[ F\colon =\{t> 0 : T_{C(t)} \,\,\text{is factorizable}\}\,. \]
We will show that there exists $t_0> 0$ such that $F\cap (0,
t_0)=\emptyset\,.$

By Lemma \ref{lem222222}, for each $t\in F$ we can find a finite von
Neumann algebra $N(t)$ with normal faithful tracial state
$\tau_{N(t)}$ and four unitary operators $u_1(t)\,, \ldots\,,
u_4(t)\in N(t)$ such that
\[ C(t)_{jk}=\tau_{N(t)}(u_j(t)^*u_k(t))\,, \quad 1\leq j, k\leq 4\,. \]
Since $\omega+\overline{\omega}+1=0$ and
$\omega^2=\overline{\omega}$\,, we can express $u_1(t)\,, \ldots \,,
u_4(t)$ in the form $u_1(t)=x(t)+w(t)$\,,
$u_2(t)=x(t)+y(t)+z(t)$\,, $u_3(t)=x(t)+\omega
y(t)+\overline{\omega} z(t)$\,, $u_4(t)=x(t)+\overline{\omega}
y(t)+\omega z(t)$\,, where $x(t)\colon =(u_2(t)+u_3(t)+u_4(t))/3$\,,
$y(t)\colon =(u_2(t)+\overline{\omega}u_3(t)+\omega u_4(t))/3$\,,
$z(t)\colon =(u_2(t)+\omega u_3(t)+\overline{\omega} u_4(t))/3$\,,
$w(t)=x(t)-u_1(t))$\,.
Note that for all $t\in F$\,,
\begin{eqnarray}\label{eq2929}
\|x(t)\|\leq 1\,, \quad \|y(t)\|\leq 1\,, \quad \|z(t)\|\leq 1\,.
\end{eqnarray}
We prove next that
\begin{eqnarray}\label{eq29296}
\|y(t)\|_2\leq 2\|z(t)\|_2+\|w(t)\|_2\,.
\end{eqnarray}
For this, observe first that $x(t)^*y(t)+y(t)^*z(t)+z(t)^*x(t)= 0$\,, which implies that
\begin{eqnarray}\label{eq29298}
\|x(t)^*y(t)\|_2&\leq &\|y(t)^*z(t)\|_2+\|z(t)^*x(t)\|_2\\
&\leq & \|y(t)\|\|z(t)\|_2+\|z(t)\|_2\|x(t)\|\nonumber\\&\leq &
\,\,2\|z(t)\|_2\,,\nonumber
\end{eqnarray}
wherein we used (\ref{eq2929}).
Now recall that $x(t)=u_1(t)+w(t)$\,, where $u_1(t)$ is a unitary. Thus, by
(\ref{eq29298}), \begin{eqnarray*}
\|y(t)\|_2\,\,=\,\,\|u_1(t)^*y(t)\|_2&=&\|x(t)^*y(t)-w(t)^*y(t)\|_2\\&\leq
&
\|x(t)^*y(t)\|_2+\|w(t)\|_2\|y(t)\|_2\nonumber\\
&\leq &2\|z(t)\|_2+\|w(t)\|_2\,,\nonumber
\end{eqnarray*}
which proves (\ref{eq29296}).
Next, observe that the $2$-norms of $y(t)$\,, $z(t)$ and
$w(t)$ can be expressed in terms of the entries of the matrix $L$, because for $c:=(c_1\,, c_2\,, c_3\,, c_4)\in
\mathbb{C}^4$ and $t\in F$ we have \[ \left\|\sum\limits_{j=1}^4
c_ju_j(t)\right\|_2^2=\sum\limits_{j, k=1}^4 \bar{c_j} c_k
\tau_N(u_j(t)^*u_k(t))=\sum\limits_{j, k=1}^4 \bar{c_j} c_k
e^{-tL_{jk}}=f(c, t)\,, \] where the function $t\mapsto f(c, t)$ is actually defined for
all $t\geq 0$ and satisfies
\begin{eqnarray}\label{eq2929801}
f(c, t)=|c_1+c_2+c_3+c_4|^2-\left(\sum\limits_{j, k=1}^4 \bar{c_j}
c_k L_{jk}\right)t+O(t^2)\,, \quad \text{as}\,\, t\searrow
0\,,
\end{eqnarray}
in Landau's O-notation. Consider now the three special cases where
$c\colon =(c_1\,, c_2\,, c_3\,, c_4)\in \mathbb{C}^4$ is equal to
$(0, 1/3, \overline{\omega}/3, \omega/3)$\,, $(0, 1/3, \omega/3,
\overline{\omega}/3)$ and $(-1, 1/3, 1/3, 1/3)$\,, respectively. For $t\geq 0$\,, let us denote the function $f(c, t)$ in (\ref{eq2929801}) by $g(t)$\,, $h(t)$ and $k(t)$\,, respectively, in each of the corresponding case. Then,
\begin{eqnarray}\label{eq292980}
g(t)=\|y(t)\|_2^2\,, \quad h(t)=\|z(t)\|_2^2\,,
\quad k(t)=\|z(t)\|_2^2\,, \qquad t\in F\,.
\end{eqnarray}
Moreover,
by (\ref{eq2929801})\,,
\begin{eqnarray}\label{eq2929808}
g(t)=t+O(t^2)\,, \quad h(t)=O(t^2)\,, \quad k(t)=O(t^2)\,, \qquad
\text{as} \,\, t\searrow 0\,.
\end{eqnarray}
Assume now that $\inf(F)=0$\,. Then there exists a sequence
$(t_n)_{n\geq 1}$ in $F$ such that $t_n\rightarrow 0$ as $n\rightarrow \infty$\,.
By (\ref{eq29296}) and (\ref{eq292980}) we have
\[ g(t_n)^{1/2}\leq 2h(t_n)^{1/2}+k(t_n)^{1/2}\,, \quad n\geq 1\,. \]
However, by (\ref{eq2929808})\,, $2h(t_n)^{1/2}+k(t_n)^{1/2}=O(t_n)$\,, while
$g(t_n)^{1/2}=t_n^{1/2}+O(t_n^{3/2})$\,, both for large enough $n$\,. This gives rise to a contradiction. Hence $\inf(F)> 0$\,, i.e.,
there exists $t_0> 0$ such that $(0, t_0)\cap F=\emptyset$\,. The
proof is complete.
\end{proof}

\begin{rem}\label{rem6666777990}\rm
The above theorem is to be contrasted with a result of K\"ummerer and Maassen (cf. \cite{KM}), showing that if
$(T(t))_{t\geq 0}$
is a one-parameter semigroup of $\tau_n$-Markov maps on $M_n(\mathbb{C})$ satisfying $T(t)^*=T(t)$\,, for all $t\geq 0$\,, then
$T(t)\in \text{conv}(\text{Aut}(M_n(\mathbb{C})))$\,, for all $t\geq 0$\,.
In particular, $T(t)$ is factorizable, for all $t\geq 0$\,.

In very recent work, Junge, Ricard and Shlyakhtenko \cite{JRS} have generalized K\"ummerer and Maassen's result to the case of a strongly continuous one-parameter semigroup $(T(t))_{t\geq 0}$ of self-adjoint Markov maps on an arbitrary von Neumann algebra with a faithful, normal tracial state by proving that also in this case $T(t)$ is factorizable, for all $t\geq 0$\,. This result has been independently obtained by Y. Dabrowski (see \cite{Dab}).
\end{rem}

\begin{rem}\label{rem6666777}\rm
In the recent preprint \cite{DyJ}, K. Dykema and K. Juschenko have indirectly exhibited an example of a $\tau_4$-Markov map on $M_4(\mathbb{C})$ which is not factorizable.

More precisely, for every $n\geq 1$ they considered the sets
${\mathcal F}_n$\,, defined as the closure of the union over $k\geq 1$ of sets of $n\times n$ complex matrices $(b_{ij})_{1\leq i, j\leq n}$ such that  $b_{ij}=\tau_k(u_iu_j^*)$\,, where $u_1, \ldots, u_n\in {\mathcal U}(k)$\,, respectively,
${\mathcal G}_n$\,, consisting of all $n\times n$ complex matrices $(b_{ij})_{1\leq i, j\leq n}$ such that $b_{ij}=\tau_M(u_iu_j^*)$\,, where $u_1, \ldots, u_n$ are unitaries in some von Neumann algebra $M$ equipped with normal faithful tracial state $\tau_M$ (where $M$ varies).
By a refinement of Kirchberg's deep results from \cite{Ki}, they concluded that Connes' embedding problem whether every $II_1$-factor with separable predual embeds in the ultrapower of the hyperfinite $II_1$ factor has an affirmative answer, if and only if ${\mathcal F}_n={\mathcal G}_n$\,,  for all $n\geq 1$\,. Further, they pointed out that ${\mathcal F}_n\subseteq {\mathcal G}_n\subseteq \Theta_n$\,, for all $n\geq 1$\,, where $\Theta_n$ is the set of $n\times n$ (complex) correlation matrices, i.e., positive semi-definite matrices whose entries on the main diagonal are all equal to 1.
A natural question to consider is whether ${\mathcal F}_n=\Theta_n$\,, for all $n\geq 1$. One of the results of \cite{DyJ} is that the answer to this question is negative, as soon as $n\geq 4$\,. More precisely, Dykema and Juschenko showed that ${\mathcal G}_4$ has no extreme points of rank 2, while there are extreme points of rank 2 in $\Theta_4$\,. Hence ${\mathcal G}_4\neq \Theta_4$\,.
In view of Lemma \ref{lem222222} above, any element of $\Theta_4\setminus {\mathcal G}_4$ is an example of a non-factorizable $\tau_4$-Markov map on $M_4(\mathbb{C})$\,.
\end{rem}

\section{Kummerer's notions of dilation and their connection to factorizability}

The following definitions are due to K\"ummerer (see \cite{Ku3}, Definitions 2.1.1 and 2.2.4, respectively):

\begin{defi}\label{def456666}
Let $M$ be a von Neumann algebra with a normal faithful state $\phi$ and let $T\colon M\rightarrow M$ be a $\phi$-Markov map. A dilation of $T$ is a quadruple $(N, \psi, \alpha, \iota)$\,, where $N$ is a von Neumann algebra with a normal faithful state $\psi$\,, $\alpha\in \text{Aut}(N, \psi)$\,, i.e., $\alpha$ is an automorphism of $N$ leaving $\psi$ invariant and $\iota\colon M\rightarrow N$ is a $(\phi, \psi)$-Markov $*$-monomorphism, satisfying
\begin{equation}\label{eq1111111111111111112}
T^n=\iota^*\circ \alpha^n\circ \iota\,, \quad n\geq 1\,.
\end{equation}
Furthermore, we say that that $(N, \psi, \alpha, \iota)$ is a dilation of $T$ of order 1 if (\ref{eq1111111111111111112}) holds for $n=1$\, but not necessarily for $n\geq 2$\,.
\end{defi}

\begin{defi}\label{def45666669}
A dilation $(N, \psi, \alpha, \iota)$ of a $\phi$-Markov map $T\colon M\rightarrow M$ is called a Markov dilation if
\begin{equation}\label{eq11111111111111111123}
\mathbb{P}_{\{0\}}(x)=\mathbb{P}_{(-\infty, 0]}(x)\,, \quad x\in (\bigcup_{k\geq 0} \alpha^k\circ {\iota}(M))^{\prime\prime}\,,
\end{equation}
where for $I\subseteq \mathbb{Z}$\,, $\mathbb{P}_I$ denotes (cf. Lemma 2.1.3 of \cite{Ku3}) the unique $\psi$-preserving normal faithful conditional expectation of $N$ onto its subalgebra $\left(\bigcup_{k\in I} \alpha^k\circ {\iota}(M)\right)^{\prime\prime}$\,.
\end{defi}

\begin{rem}\label{rem899999999}\rm
The condition (\ref{eq11111111111111111123}) is equivalent to
\begin{equation*}
\mathbb{P}_{\{0\}} \mathbb{P}_{[0, \infty)}=\mathbb{P}_{(-\infty, 0]} \mathbb{P}_{[0, \infty)}\,.
\end{equation*}
Clearly, $\mathbb{P}_{[0, \infty)} \mathbb{P}_{\{0\}}=\mathbb{P}_{\{0\}}$\,, and since both $\mathbb{P}_{\{0\}}$ and $\mathbb{P}_{[0, \infty)}$ extend uniquely to self-adjoint projections on $L^2(N\,, \psi)$\,, it also follows that $\mathbb{P}_{\{0\}} \mathbb{P}_{[0, \infty)}=\mathbb{P}_{\{0\}}$\,. Hence (\ref{eq11111111111111111123}) is further equivalent to
\begin{equation}\label{eq33333333333333356}
\mathbb{P}_{\{0\}}=\mathbb{P}_{(-\infty, 0]} \mathbb{P}_{[0, \infty)}\,.
\end{equation}
\end{rem}

In the Spring of 2008, C. Koestler informed us in a private correspondence that he was aware of the fact that for a given Markov map, the existence of a dilation (in the sense of Definition \ref{def456666}) is actually equivalent to C. Anantharaman-Delaroche's condition of factorizability of the map. The proof relies on a construction of inductive limit of von Neumann algebras naturally associated to a Markov $*$-monomorphism, studied by K\"ummerer in his unpublished Habilitationsschrift \cite{Ku1}. We are very grateful to C. Koestler for sharing all this information with us and for kindly providing us with a copy of \cite{Ku1}.

We were further able to show that the existence of a dilation for a given Markov map is equivalent to the existence of a Markov dilation (in the sense of Definition \ref{def45666669}) for it.
For completeness, we collect together all these equivalent statements in the following theorem.

\begin{theorem}\label{th77777777232323}
Let $M$ be a von Neumann algebra with normal faithful state $\phi$ and let $T\colon M\rightarrow M$ be a $\phi$-Markov map. The following statements are equivalent:
\begin{enumerate}
\item [$(1)$] $T$ is factorizable.
\item [$(2)$] $T$ has a dilation.
\item [$(3)$] $T$ has a Markov dilation.
\end{enumerate}
\end{theorem}

For convenience, we include the details of the above-mentioned inductive limit construction for von Neumann algebras, that will be used in the proof of Theorem \ref{th77777777232323}.

\begin{lemma}\label{lem7000000000000}
Suppose that for each positive integer $k$\,, we are given a von Neumann algebra $M_k$ with a normal faithful state $\psi_k$ and a unital $*$-monomorphism $\beta_k \colon M_k \to M_{k+1}$ such that $\psi_{k+1} \circ \beta_k = \psi_k$\,, satisfying, moreover,
\begin{equation}\label{eq4000000000000004}
\sigma_t^{\psi_{k+1}} \circ \beta_k=\beta_k\circ \sigma_t^{\psi_k}\,, \quad t\in \mathbb{R}\,.
\end{equation}
Then, there exists a von Neumann algebra $M$ with a normal faithful state $\psi$ and unital $*$-monomorphisms $\mu_k \colon M_k \to M$ such that $\mu_{k+1} \circ \beta_k = \mu_k$\,, $\psi \circ \mu_k = \psi_k$ and $\sigma_t^\psi \circ \mu_k=\mu_k\circ \sigma_t^{\psi_k}$\,, $t\in \mathbb{R}$\,, for all $k\geq 1$, and such that $\bigcup_{k=1}^\infty \mu_k(M_k)$ is weakly dense in $M$.

Moreover, if we are given another von Neumann algebra $N$ with a normal faithful state $\varphi$ and (normal) $*$-monomorphisms $\lambda_k \colon M_k \to N$ such that $\lambda_{k+1} \circ \beta_k = \lambda_k$ and $\varphi \circ \lambda_k = \psi_k$ for all $k\geq 1$, then there exists a unique $*$-monomorphism $\lambda \colon M \to N$ such that $\lambda \circ \mu_k = \lambda_k$ for all $k\geq 1$ and $\varphi \circ \lambda = \psi$.
\end{lemma}

\begin{proof}
As a first step towards the existence of $M$, let $M_\infty$ be the $C^*$-algebra inductive limit of the sequence $M_1 \to M_2 \to M_3 \to \cdots$\,. This is a $C^*$-algebra equipped with $^*$-monomorphisms $\widetilde{\mu}_k \colon M_k \to M_\infty$ (which are unital when the connecting mappings $\beta_k$ all are unital) satisfying $\widetilde{\mu}_{k+1} \circ \beta_k = \widetilde{\mu}_k$ for all $k\geq 1$\,, and which contains $\bigcup_{k=1}^\infty \widetilde{\mu}_k(M_k)$ as a norm-dense sub-algebra. The states $\widetilde{\psi}_k$ on $\widetilde{\mu}_k(M_k)$, defined by $\widetilde{\psi}_k \circ \widetilde{\mu}_k = \psi_k$, are coherent, i.e., the restriction of $\widetilde{\psi}_{k+1}$ to $\widetilde{\mu}_k(M_k)$ is equal to $\widetilde{\psi}_k$, and so they extend to a state $\widetilde{\psi}$ on $M_\infty$. Let $M$ be the weak closure of $M_\infty$ in the GNS-representation of $M_\infty$ with respect to the state $\widetilde{\psi}$, and let $\mu_k \colon M_k \to M$ be the composition of $\widetilde{\mu}_k$ with the inclusion mapping of $M_\infty$ into $M$. The state $\widetilde{\psi}$ extends to a normal state $\psi$ on $M$. It is clear that \begin{equation*} \mu_{k+1} \circ \beta_k = \mu_k\,, \quad \psi \circ \mu_k = \psi_k\,, \qquad k\geq 1\,.\end{equation*}
We prove next that $\psi$ is faithful. For simplicity, we will now identify $M_k$ with $\mu_k(M_k)$\,, $k\geq 1$\,. Then we have the inclusions
$M_1\subseteq M_2\subseteq M_3\subseteq \ldots$ and $\bigcup_{k=1}^\infty M_k$ is weakly dense in $M$\,. Moreover, these inclusions extend to isometric embeddings of the corresponding GNS Hilbert spaces for $(M_k\,, \psi_k)$\,, $k\geq 1$\,, i.e.,
$L^2(M_1\,, \psi_1)\subseteq L^2(M_2\,, \psi_2)\subseteq L^2(M_3\,, \psi_3)\subseteq \ldots$
and $\bigcup_{k=1}^\infty L^2(M_k\,, \psi_k)$ is dense in $L^2(M\,, \psi)$\,. After these identifications, for every $k\geq 1$\,, the condition $\sigma_t^{\psi_{k+1}}\circ \beta =\beta\circ \sigma_t^{\sigma_k}$\,, $t\in \mathbb{R}$\,,
is equivalent to
\begin{equation}\label{eq33300000} \sigma_t^{\psi_{k+1}}(x)=\sigma_t^{\psi_k}(x)\,, \quad x\in M_k\,, \quad t\in \mathbb{R}\,.
\end{equation}
Hence, by \cite{Ta1} there exist unique normal faithful conditional expectations $\mathbb{E}_k\colon M_{k+1}\rightarrow M_k$ such that $\psi_k\circ \mathbb{E}_k=\psi_{k+1}$\,, $k\geq 1$\,. Also from \cite{Ta1} it follows that the isometric modular conjugation operator $J_{\psi_k}$ on $L^2(M_k\,, \psi_k)$ is the restriction of $J_{\psi_{k+1}}$ to $L^2(M_k\,, \psi_k)$\,. Hence, there exists a conjugate-linear isometric involution $J$ on $L^2(M, \psi)$ which extends al the $J_k$'s\,. Since, moreover, $J_{\psi_k} \pi_{\psi_k}(M_k) J_{\psi_k}={\pi_{\psi_k} (M_k)}^\prime$\,, for all $k\geq 1$\,,
it follows that
\begin{equation}\label{eq2000786}
J \pi_\psi(M) J\subseteq {{\pi_\psi}(M)}^\prime\,.
\end{equation}
Let $\xi_{\psi}$ be the cyclic vector in $L^2(M, \psi)$ corresponding to the unit operator $1$ in $M$\,. Then $J \xi_\psi=\xi_\psi$, and therefore by (\ref{eq2000786}), $\xi_\psi$ is also cyclic for ${{\pi_\psi}(M)}^\prime$. Hence $\xi_\psi$ is separating for ${\pi_\psi}(M)$\,, which proves that $\psi$ is faithful.
Finally, in order to prove that $\sigma_t^\psi \circ \mu_k=\mu_k\circ \sigma_t^{\psi_k}$\,, $t\in \mathbb{R}$\,, $k\geq 1$\,, we have to show that under the above identifications,
\begin{equation}\label{eq224422000}
\sigma_t^\psi(x)=\sigma_t^{\psi_k}(x)\,, \quad x\in M_k\,, \quad t\in \mathbb{R}\,, \quad k\geq 1\,.
\end{equation}
By (\ref{eq33300000}), together with the fact that $\psi_k={\psi_{k+1}}_{\vert_{M_k}}$\,, $k\geq 1$\,, it follows that the modular automorphism groups $(\sigma_t^{\psi_k})_{t\in \mathbb{R}}$\,, $k\geq 1$ have a unique extension to a strongly continuous one-parameter group of automorphisms $(\sigma_t)_{t\in \mathbb{R}}$ on $M$\,. Moreover, $\psi$ satisfies the KMS condition with respect to $(\sigma_t)_{t\in \mathbb{R}}$\,, since each $\psi_k$ is a $(\sigma_t^{\psi_k})_{t\in \mathbb{R}}$-KMS state on $M_k$ (see Theorem 1.2, Chap. VIII, in \cite{Tk}). Therefore $\sigma_t=\sigma_t^\psi$\,, $t\in \mathbb{R}$\,, which proves (\ref{eq224422000}).

To prove the second part of the lemma we can without loss of generality assume that $\bigcup_{k=1}^\infty \lambda_k(M_k)$ is weakly dense in $N$ (otherwise replace $N$ by the weak closure of $\bigcup_{k=1}^\infty \lambda_k(M_k)$). Consider the GNS-representations of $M$ and $N$ on Hilbert spaces $H$ and $H'$ with respect to the normal faithful states $\psi$ and $\varphi$, respectively. Then there exist cyclic and separating vectors $\xi \in H$ and $\xi' \in H'$ for $M$ and $N$, respectively, such that \begin{equation*}\psi(x) = \langle x \xi , \xi \rangle\,, \quad \varphi(y) = \langle y \xi' , \xi' \rangle\,, \qquad  x \in M\,, y \in N\,.\end{equation*}
By the universal property of the $C^*$-algebra inductive limit, there is a $*$-monomorphism $\widetilde{\lambda} \colon M_\infty \to N$ satisfying $\widetilde{\lambda} \circ \mu_k = \lambda_k$ for all $k\geq 1$. Observe that $\varphi \circ \widetilde{\lambda}(x) = \psi(x)$\,, for all $x \in M_\infty$. It follows that the map $u_0 \colon M_\infty \xi \to H'$ defined by $u_0 x\xi = \widetilde{\lambda}(x) \xi'$\,, $x \in M_\infty$\,, is isometric and has dense range in $H'$. Hence it extends to a unitary $u \colon H \to H'$. We see that $uxu^* \xi' = \widetilde{\lambda}(x) \xi'$, and hence that $uxu^* = \widetilde{\lambda}(x)$,  for all $x \in M_\infty$. The map $\lambda \colon M \to N$ defined by $\lambda(x) = uxu^*$\,, for $x \in M$, has the desired properties.
\end{proof}

\begin{rem}\label{re6700005}\rm
We would like to draw the reader's attention upon the subtle fact that condition (\ref{eq4000000000000004}) is crucial for guaranteeing that the canonical state $\widetilde{\psi}$ on the $C^*$-algebra inductive limit $M_\infty$ extends to a {\em faithful} state $\psi$ on the von Neumann algebra $M$, obtained via the GNS representation of $M_\infty$ with respect to $\widetilde{\psi}$\,.
\end{rem}

\noindent The von Neumann algebra $(M,\psi)$ is said to be the von Neumann algebra inductive limit of the sequence
$$\xymatrix{(M_1,\psi_1) \ar[r]^-{\beta_1} & (M_2,\psi_2) \ar[r]^-{\beta_2} & (M_3,\psi_3) \ar[r]^-{\beta_3} & \cdots}.
$$

The following result is a reformulation of Proposition 2.1.7 in \cite{Ku1}\,.

\begin{prop}\label{prop80000000000}
Let $N$ be a von Neumann algebra with a normal faithful state $\psi$\,, and let $\beta\colon N\rightarrow N$ be a $\psi$-Markov $*$-monomorphism. Then there exists a von Neumann algebra $\widetilde{N}$ with a normal faithful state $\widetilde{\psi}$\,, a $(\psi, \widetilde{\psi})$-Markov embedding $\iota\colon N\rightarrow \widetilde{N}$ and an $\alpha\in \text{Aut}(\tilde{N})$ for which $\widetilde{\psi}\circ \alpha=\widetilde{\psi}$ such that $\beta=\iota^*\circ \alpha \circ \iota$\,.
\end{prop}

\begin{proof}
Let $(\widetilde{N}, \widetilde{\psi})$ be the von Neumann algebra inductive limit of the sequence
$$\xymatrix{(N,\psi) \ar[r]^-{\beta} & (N,\psi) \ar[r]^-{\beta} & (N,\psi) \ar[r]^-{\beta} & \cdots},
$$
and for every $k\geq 1$\,, let $\mu_k \colon N \to \widetilde{N}$ be the associated $*$-monomorphism from the $k^{\text{th}}$ copy of $N$ into $\widetilde{N}$.
By the second part of Lemma \ref{lem7000000000000} applied to the $*$-monomorphisms $\lambda_k \colon N \to \widetilde{N}$ given by $\lambda_k = \mu_k \circ \beta$, there exists a $*$-monomorphism $\alpha$ on $\widetilde{N}$ such that $\alpha \circ \mu_{k} = \mu_k \circ \beta$\,, for all $k\geq 1$.  It follows that
$$\bigcup_{k=1}^\infty \mu_k(N) = \bigcup_{k=1}^\infty \mu_{k+1} \circ \beta(N) =\bigcup_{k=1}^\infty \alpha \circ \mu_{k+1}(N) \subseteq {\mathrm{Im}}(\alpha).$$
As  $\bigcup_{k=1}^\infty \mu_k(N)$ is dense in $\widetilde{N}$ and the image of $\alpha$ is a von Neumann subalgebra of $\widetilde{N}$, this shows that $\alpha$ is onto, and hence an automorphism.

Take $\iota \colon N \to \widetilde{N}$ to be $\mu_1$. Then $\alpha \circ \iota = \iota \circ \beta$. Moreover, by Lemma \ref{lem7000000000000}, $\iota$ is a $\psi$ -Markov map. Since \[ \widetilde{\psi}\circ \alpha\circ \mu_k=\widetilde{\psi}\circ \mu_k\circ \beta=\psi\circ \beta=\psi=\widetilde{\psi}\circ \mu_k\,, \quad k\geq 1\,, \]
$\widetilde{\psi}\circ \alpha$ and $\widetilde{\psi}$ coincide  on $\bigcup_{k\geq 1}^\infty \mu_k(N)$\,. Therefore, $\widetilde{\psi}=\widetilde{\psi}\circ \alpha$\,.
The existence of the adjoint map $\iota^* \colon \widetilde{N} \to N$ follows from Remark \ref{rem78787777754}. As $\iota^* \circ \iota$ is the identity on $N$ we get that $\beta = \iota^* \circ \alpha \circ \iota$ using the previously obtained identity $\iota \circ \beta = \alpha \circ \iota$. This completes the proof.
\end{proof}

{\bf Proof of Theorem \ref{th77777777232323}}:
The implication $(3)\Rightarrow (2)$ is trivial. Also, the implication $(2)\Rightarrow (1)$ follows immediately, since if $(N, \psi, \alpha, \iota)$ is a dilation of $T$\,, then $T=\iota^* \circ (\alpha\circ \iota)$
is a factorization of $T$ through $(N, \psi)$ in the sense of Definition \ref{defi898}, because $\alpha\in \text{Aut}(N, \psi)$ implies that $\alpha\circ \sigma_t^\psi=\sigma_t^\psi\circ \alpha$\,, $t\in \mathbb{R}$\,.

Next we prove that $(1)\Rightarrow (3)$\,. Assume that $T$ is factorizable. Then, by Theorem 6.6 in \cite{AD}, there exists a von Neumann algebra $N$ with a normal faithful state $\psi$\,, a $\psi$-Markov normal $*$-endomorphism $\beta\colon N\rightarrow N$ and a $(\phi, \psi)$-Markov $*$-monomorphism $j\colon M\rightarrow N$ such that $T^n=j^*\circ \beta^n \circ j$\,, for all $n\geq 1$\,. By Proposition \ref{prop80000000000}, we can find a dilation $(\tilde{N}\,, \tilde{\psi}\,, \alpha\,, \iota)$ of $\beta$\,, where $\alpha\in \text{Aut}(\tilde{N}\,, \tilde{\psi})$\,. We may (and will) consider $N$ as a subalgebra  of $\tilde{N}$\,. In this way, $\iota$ is just the inclusion map, $\iota^*$ is the $\tilde{\psi}$-preserving normal faithful conditional expectation of $\tilde{N}$ onto $N$\,, and $\beta=\alpha_{\vert_N}$\,. Then it is clear that with $\widetilde{j}\colon = \iota\circ j$, the quadruple $(\widetilde{N}, \widetilde{\psi}, \alpha, \widetilde{j})$ is a dilation of $T$ (which actually proves $(1)\Rightarrow (2)$). To complete the proof of the implication $(1)\Rightarrow (3)$, we will show that by the construction of $\beta$ from \cite{AD}, the quadruple
$(\widetilde{N}, \widetilde{\psi}, \alpha, \widetilde{j})$ becomes a Markov dilation of $T$\,.

For $J\subseteq \{n\in \mathbb{Z}: n\geq 0\}$\,, let $\mathbb{E}_J$ denote the unique $\psi$-preserving conditional expectation of $N$ onto its subalgebra ${\mathcal B}_J\colon = (\bigcup_{k\in J} \beta^k \circ j(M))^{\prime\prime}$\,. Then, by condition $(6.2)$ in Theorem 6.6 of \cite{AD}, we get
\begin{equation}\label{eq5555555555555543}
\mathbb{E}_{[0, n+k]}\circ \beta^k=\beta^k \circ \mathbb{E}_{[0, n]}\,, \quad n, k\geq 0\,.
\end{equation}
Moreover, with ${\mathcal B}$ and $({\mathcal B}_n\,, \phi_n)$\,, $n\geq 0$ defined as in the proof of the above-mentioned theorem, we have ${\mathcal B}_n={\mathcal B}_{[0, n]}$  and $N={\mathcal B}={\mathcal B}_{[0, \infty)}$\,. Hence $\mathbb{E}_{[0, n]}=\mathbb{E}_{{\mathcal B}_n}$\,, the unique $\psi$-invariant conditional expectation of $N$ onto ${\mathcal B}_n$\,, and $\mathbb{E}_{[0, \infty)}=\mathbb{E}_N=\text{id}_N$\,.
Set now $H\colon = L^2(N, \psi)$\,, $H_n\colon = L^2({\mathcal B}_n\,, \phi_n)$\,, $n\geq 0$\,, and let $V\colon = \widetilde{\beta}$\,, be the unique extension of $\beta$ to an isometry on $L^2(N, \psi)$\,. By (\ref{eq5555555555555543}),
\begin{equation*}
P_{H_{n+k}} V^k=V^k P_{H_n}\,, \quad n, k\geq 0\,,
\end{equation*}
where $P_K\in {\mathcal B}(H)$ denotes the orthogonal projection onto a closed subspace $K$ of $H$\,. Hence,
\begin{equation}\label{eq22222222222222227}
P_{H_{n+k}} P_{V^k(H)}=P_{H_{n+k}} V^k (V^*)^k=V^k P_{H_n}(V^*)^k=P_{V^k(H_n)}\,, \quad n, k\geq 0\,.
\end{equation}
By the definition, it is clear that $\beta^k({\mathcal B}_J)={\mathcal B}_{J+k}$\,, for all $k\geq 0$ and all $J\subseteq \{n\in \mathbb{Z}: n\geq 0\}$\,. In particular,
\begin{equation*}
\beta^k(N)=\beta^k({\mathcal B}_{[0, \infty)})={\mathcal B}_{[k, \infty)}\,, \quad \beta^k({\mathcal B}_n)={\mathcal B}_{[k, k+n]}\,, \qquad n, k\geq 0\,.
\end{equation*}
Thus, by restricting (\ref{eq22222222222222227}) to $N\subseteq L^2(N, \psi)$\,, we get
$\mathbb{E}_{[0, n+k]} \mathbb{E}_{[k, \infty)}=\mathbb{E}_{[k, k+n]}$\,, for all $n, k\geq 0$\,.
In particular, we have
\begin{equation}\label{eq22222222222233322}
\mathbb{E}_{[0, k]} \mathbb{E}_{[k, \infty)}=\mathbb{E}_{\{k\}}\,, \quad k\geq 0\,.
\end{equation}
Since $\widetilde{j}=\iota\circ j$\,, we have from Proposition \ref{prop80000000000} that
\begin{equation*}
\alpha^k \circ \tilde{j}(M)=\beta^k \circ j(M)\,, \quad k\geq 0\,.
\end{equation*}
Hence, by composing (\ref{eq22222222222233322}) from the right with the $\widetilde{\psi}$-preserving conditional expectation $\iota^*$ of $\tilde{N}$ onto $N$, we get (following the notation set forth in Definition \ref{def45666669}) that
\begin{equation}\label{eq101010}
\mathbb{P}_{[0, k]} \mathbb{P}_{[k, \infty)}=\mathbb{P}_{\{k\}}\,, \quad k\geq 0\,.
\end{equation}
Note that for every $I\subseteq \mathbb{Z}$\,, by the definition of $\mathbb{P}_I$ one has $\alpha^n \mathbb{P}_I \alpha^{-n}=\mathbb{P}_{I+n}$\,, for all $n\in \mathbb{Z}$\,. Hence, from (\ref{eq101010}) we get that
\begin{equation*}
\mathbb{P}_{[-n, 0]} \mathbb{P}_{[0, \infty)}=\mathbb{P}_{\{0\}}\,, \quad n\geq 0\,.
\end{equation*}
In the limit as $n\rightarrow \infty$\,, this yields
\begin{equation*}
\mathbb{P}_{(-\infty, 0]} \mathbb{P}_{[0, \infty)}=\mathbb{P}_{\{0\}}\,,
\end{equation*}
a condition which, by Remark \ref{rem899999999}, ensures that $(\tilde{N}\,, \tilde{\psi}\,, \alpha\,, \tilde{j})$ is a Markov dilation of $T$\,.\qed

Note that by the proof of Theorem \ref{th77777777232323} it follows that a $\phi$-Markov map $T\colon M\rightarrow M$ admits a dilation if and only if it has a dilation of order 1 (see Definition \ref{def456666}), since in order to show that $(2)\Rightarrow (1)$ above we have only used the existence of a dilation of order 1 for the given map.

\begin{rem}\label{rem8080}\rm
K\"ummerer constructed in \cite{Ku1} examples of $\tau_n$-Markov maps on $M_n(\mathbb{C})$\,, $n\geq 3$\,, having no dilation, as follows:

\noindent $(1)$ Let $a_1={\frac1{\sqrt{2}}}\left(
\begin{array}
[c]{ccc}%
0 & 0 & 0\\
1 & 0 & 0\\
0 & 1 & 0
\end{array}
\right)$\,, $a_2={\frac1{\sqrt{2}}}\left(
\begin{array}
[c]{ccc}%
0 & 1 & 0\\
0 & 0 & 1\\
0 & 0 & 0
\end{array}
\right)$\,, and $a_3={\frac1{\sqrt{2}}}\left(
\begin{array}
[c]{ccc}%
0 & 0 & 1\\
0 & 0 & 0\\
1 & 0 & 0
\end{array}
\right)$\,.
Then the map given by
$Tx\colon= \sum_{i=1}^3 a_i^* x a_i$\,, $x\in
M_3(\mathbb{C})$\,, is a $\tau_3$-Markov map having no dilation.\\[0.1cm]
\noindent $(2)$  Let $n\geq 4$ and consider the $n\times n$ diagonal matrices
$a_1=\text{diag}(1\,, 1/{\sqrt{2}}\,, 1/{\sqrt{2}}\,, 0\,, \ldots\,, 0)$ and $a_2=\text{diag}(0\,, 1/{\sqrt{2}}\,, i/{\sqrt{2}}\,, 1\,, \ldots\,, 1)$.
Then the map given by $Tx\colon= \sum_{i=1}^2 a_i^* x a_i$\,, $x\in
M_n(\mathbb{C})$\,, is a $\tau_n$-Markov Schur multiplier with no dilation.

In view of Koestler's communication, these are all examples of non-factorizable Markov maps. In \cite{Ku1}, K\"ummerer also constructed an example of a $\tau_6$-Markov Schur multiplier on $M_6(\mathbb{C})$ which admits a dilation, hence it is factorizable, but does not lie in $\text{conv}(\text{Aut}(M_6(\mathbb{C})))$. However, he did not consider the one-parameter semigroup case.

We were also informed by B. V. R. Bhat and A. Skalski \cite{BhS} that, unaware of K\"ummerer's examples and their connection with Anantharaman-Delaroche's problem, as well as of our already existing work, they have also constructed examples of a non-factorizable $\tau_3$-Markov map on $M_3(\mathbb{C})$\,, respectively, of a $\tau_4$-Markov Schur multiplier on $M_4(\mathbb{C})$ which is not factorizable.
\end{rem}

\section{The noncommutative Rota dilation property}

The following was introduced in \cite{JMX} (see Definition 10.2 therein):

\begin{defi}\label{defrotadil}
Let $M$ be a von Neumann algebra equipped with a normalized (normal and faithful) trace $\tau$. We say that a bounded operator $T\colon M\rightarrow M$ satisfies the Rota dilation property if there exists a von Neumann algebra $N$ equipped with a normalized (normal and faithful) trace $\tau_N$\,, a normal unital faithful $*$-representation $\pi\colon M\rightarrow N$ which preserves the traces (i.e., $\tau_N\circ \pi=\tau$), and a decreasing sequence $(N_m)_{m\geq 1}$ of von Neumann subalgebras
of $N$ such that
\[ T^m=Q\circ \mathbb{E}_m\circ \pi\,, \quad m\geq 1\,. \]
Here $\mathbb{E}_m$ denotes the canonical (trace-preserving) conditional expectation of $N$ onto $N_m$\,, and $Q\colon N\rightarrow M$ is the conditional expectation associated to $\pi$\,, that is, $Q\colon =\pi^*=\pi^{-1}\circ \mathbb{E}_{\pi(M)}$\,, where $\mathbb{E}_{\pi(M)}$ is the trace-preserving conditional expectation of $N$ onto $\pi(M)$.
\end{defi}

\begin{rem}\label{rem102}\rm If $T\colon M\rightarrow M$ has the Rota dilation property, then $T$ is completely positive, unital and trace-preserving. Since in the tracial setting condition $(4)$ in Definition \ref{defmarkov} is trivially satisfied, it follows that $T$ is automatically a $\tau$-Markov map. Moreover, since
$\mathbb{E}_1$ (viewed as an operator from $N$ into $N$) can be written as $\mathbb{E}_1=j_1^*\circ {j_1}$\,, where $j_1:N_1\hookrightarrow N$ is the inclusion map, then
\begin{equation}\label{eq345} T= Q\circ \mathbb{E}_1\circ \pi=(j_1^*\circ \pi)^*\circ (j_1^*\circ \pi)\,. \end{equation}
Hence, $T$ is positive as an operator on the pre-Hilbert space M with inner product $\langle x, y\rangle\colon =\tau(y^* x)$\,, $x, y\in M$\,. This also implies that $T=T^*$\,, where $T^*\colon M\rightarrow M$ is the adjoint of $T$ in the sense of (\ref{eq5000567}). (This observation is also stated in \cite{JMX} (cf. Remark 10.3 therein)).

Furthermore, equalities (\ref{eq345}) show that the Rota dilation property implies factorizability of $T$, in view of Remark \ref{rem787877777} $(b)$.
\end{rem}

The following is a consequence of Theorem 6.6 in \cite{AD}:

\begin{theorem}\label{th6767}
If $T:M\rightarrow M$ is a factorizable $\tau$-Markov map with $T=T^*$, then $T^2$ has the Rota dilation property.
\end{theorem}

\begin{proof}
Since $T$ is factorizable, Theorem 6.6 in \cite{AD} ensures the existence of a von Neumann algebra $N$ with a normal faithful state $\psi$, a normal unital endomorphism $\beta: N\rightarrow N$ which is $\psi$-Markov and a normal unital $*$-homomorphism $J_0:M\rightarrow N$ which is $(\phi, \psi)$-Markov such that, if we set $J_n:=\beta^n\circ J_0$ and ${\mathbb{E}}_{[n}$ denotes the conditional expectation of $N$ onto its von Neumann subalgebra generated by $\bigcup_{k\geq n} J_k(M)$ for all $n\geq 0$\,, while ${\mathbb{E}}_{0]}$ is the conditional expectation of $N$ onto $J_0(M)$\,, then
\begin{eqnarray}\label{eq343434}
{\mathbb{E}}_{0]}\circ J_n&=& J_0\circ T^n\,, \qquad \,\,n\geq 1\,,\\
{\mathbb{E}}_{[n}\circ J_0&=& J_n\circ (T^*)^n\,, \quad n\geq 1\,.\nonumber
\end{eqnarray}
It follows that
$J_0\circ T^n\circ (T^*)^n={\mathbb{E}}_{0]}\circ J_n\circ (T^*)^n={\mathbb{E}}_{0]}\circ {\mathbb{E}}_{[n}\circ J_0$\,, $n\geq 1$\,.
Since $T^*=T$\,, we infer that
\begin{equation}\label{eq3333}
(T^2)^n=T^n\circ (T^*)^n=J_0^{-1}\circ {\mathbb{E}}_{0]}\circ {\mathbb{E}}_{[n}\circ J_0=J_0^*\circ {\mathbb{E}}_{[n}\circ J_0\,, \quad n\geq 1\,. \end{equation}
Observing that $({\mathbb{E}}_{[n})_{n\geq 1}$ is a sequence of conditional expectations with decreasing ranges, (\ref{eq3333}) shows that $T^2$ has the Rota dilation property.
\end{proof}

Note that if $M$ is abelian and $T:M\rightarrow M$ is factorizable, then, following the construction in \cite{AD}, one can choose an abelian dilation $N$ for $T$\,. Therefore, Theorem \ref{th6767} is a noncommutative analogue of Rota's classical dilation theorem for Markov operators. The next result shows that the factorizability condition cannot be removed from the hypothesis of Theorem \ref{th6767}, thus the Rota dilation theorem does not hold in general in the noncommutative setting.

\begin{theorem}\label{th5656}
There exists a $\tau_n$-Markov map $T\colon M_n(\mathbb{C})\rightarrow M_n(\mathbb{C})$\,, for some $n\geq 1$\,, such that $T=T^*$,
but $T^2$ is not factorizable. In particular, $T^2$ does not have the Rota dilation property.
\end{theorem}

To prove the theorem, we start with the following

\begin{lemma}\label{lem6767}
Let $n, d\in \mathbb{N}$ and consider $a_1\,, \ldots \,, a_d\in M_n(\mathbb{C})$ be self-adjoint with $\sum_{i=1}^da_i^2=1_n$\,. Set
\begin{equation}\label{eq22233332222} T(x)\colon =\sum\limits_{i=1}^d a_i x a_i\,, \quad x\in M_n(\mathbb{C})\,. \end{equation}
Suppose that
\begin{enumerate}
\item [$(i)$] $a_i^2a_j=a_ja_i^2$\,, $1\leq i, j\leq d$\,.
\item [$(ii)$] $A\colon=\{a_ia_j: 1\leq i, j\leq d\}$ is linearly independent.
\item [$(iii)$] $B\colon=\cup_{i=1}^6 B_i$ is linearly independent, where
\end{enumerate}
$B_1\colon=\{a_ia_ja_ka_l: 1\leq i\neq j\neq k\neq l\leq d\}$\,, $B_2\colon=\{a_ia_ja_k^2: 1\leq i\neq j\neq k\neq k\leq d\}$\,, $B_3\colon= \{a_i^3 a_j: 1\leq i\neq j\leq d\}$\,, $B_4\colon=\{a_ia_j^3: 1\leq i\neq j\leq d\}$\,, $B_5\colon=\{a_i^2a_j^2: 1\leq i< j\leq d\}$\,, $B_6\colon=\{a_i^4: 1\leq i\leq d\}$\,. Furthermore, it is assumed that $B$ is the disjoint union of the sets $B_i$\,, $1\leq i\leq 6$\,, and that the elements listed in each $B_i$ are distinct.

Assume further that
\begin{enumerate}
\item [$(iv)$] $d\geq 5$\,.
\end{enumerate}
Then $T$ is a self-adjoint $\tau_n$-Markov map, for which $T^2$ is not factorizable.
\end{lemma}

\begin{proof}
We have $T^2x=\sum_{i, j=1}^d (a_ia_j)^* x (a_i a_j)$\,, for all $x\in M_n(\mathbb{C})$\,.
It is clear that $T^2$ is a $\tau_n$-Markov map, for which Theorem \ref{th1} can be applied, due to condition $(ii)$. Hence, if $T^2$ were factorizable, it would then follow that there exists a finite von Neumann algebra $N$ with a norma faithful tracial state $\tau_N$\,, and a unitary $u\in M_n(N)$ such that
\[  T^2x \otimes 1_N=(\text{id}_{M_n(\mathbb{C})}\otimes {\tau_N})(u^*(x\otimes 1_N)u)\,, \quad x\in M_n(\mathbb{C})\,. \]
Moreover, since $T^2$ is self-adjoint, an easy argument shows that $u$ can be chosen to be self-adjoint. Namely, one can replace $N$ by $M_2(N)$ and $u$ by
$\tilde{u}:= \left(
\begin{array}
[c]{cc}%
0 & u^*\\
u & 0
\end{array}
\right)\in M_2(M_n(N))=M_n(M_2(N))$\,. Moreover, Theorem \ref{th1} ensures that $u$ is of the form
\begin{eqnarray}\label{eq345456}
u=\sum\limits_{i, j=1}^d a_i a_j \otimes v_{ij}\,,
\end{eqnarray}
where $v_{ij}\in N$\,, for all $1\leq i, j\leq d$\,, and
\begin{eqnarray}\label{eq345457}
\tau_N(v_{ij}^*v_{kl})=\delta_{ik}\delta_{jl}\,.
\end{eqnarray}
By $(i)$, the elements $v_{ij}$\,, $1\leq i, j\leq d$ are uniquely determined from (\ref{eq345456}). Since, moreover, $u=u^*$\,, we deduce that
\begin{eqnarray}\label{eq345458}
v_{ij}^*=v_{ji}\,, \quad 1\leq i, j\leq d\,.
\end{eqnarray}
Now, by condition $(i)$, we obtain the following set of relations for all $i\neq j\neq k\neq i$\,,
\begin{equation}\label{eq345459}
a_ia_ja_k^2=a_ia_k^2a_j=a_k^2a_ia_j\end{equation}
and, respectively, for all $i\neq j$\,,
\begin{equation}\label{eq8889990}
a_i^3a_j=a_ia_ja_i^2\,,\quad
a_ia_j^3=a_j^2a_ia_j\,, \quad
a_i^2a_j^2=a_ia_j^2a_i=a_j^2a_i^2=a_ja_i^2a_j\,.
\end{equation}
These conditions imply that every matrix of the form $a_i a_j a_k a_l$\,, $1\leq i, j, k, l\leq d$ occurs precisely once the set $B=\cup_{i=1}^6 B_i$\,.
Moreover, two elements of the form $b=a_i a_j a_k a_l$\,, $b^\prime=a_{i^\prime} a_{j^\prime} a_{k^\prime} a_{l^\prime}$\,, where $(i, j, k, l)\neq (i^\prime, j^\prime, k^\prime, l^\prime)$ are equal if and only if
one of the four cases listed in (\ref{eq345459}) and (\ref{eq8889990}) holds.

Furthermore, since $u^2=u^*u=1_{M_n(N)}$\,, we have
\begin{eqnarray}\label{eq889991}
1_{M_n(N)}=\sum\limits_{i, j, k, l}^d a_i a_j a_k a_l \otimes v_{ij} v_{kl}\,.
\end{eqnarray}
By applying $\text{id}_{M_n(\mathbb{C})}\otimes \tau_N$ on both sides of (\ref{eq889991}), we get
$1_n=\sum_{i, j, k, l=1}^d \tau_N(v_{ij} v_{kl})a_i a_j a_k a_l$\,,
and therefore
$0_{M_n(N)}=\sum_{i, j, k, l=1}^d a_i a_j a_k a_l \otimes (v_{ij}v_{kl}-\tau_N(v_{ij}v_{kl})1_N)$\,.
By (\ref{eq345457}) and (\ref{eq345458}), this can further be reduced to
\begin{eqnarray}\label{eq67676765}
0_{M_n(N)}=\sum\limits_{i, j, k, l=1}^d a_i a_j a_k a_l \otimes (v_{ij}v_{kl}-\delta_{il}\delta_{jk}1_N)\,.
\end{eqnarray}
Using the remark following (\ref{eq8889990}), the equation (\ref{eq67676765}) can be rewritten as
$0_{M_n(N)}=\sum_{b\in B} b\otimes w_b$\,,
where $w_b\in N$\,. Since $B$ is a linearly independent set, this implies that $w_b=0_N$\,, for all $b\in B$\,. Hence,
if $b\in B_1$\,, i.e., $b=a_i a_j a_k a_l$\,, where $i\neq j\neq k\neq l$\,, we infer that $0_N=w_b=v_{ij}v_{kl}-\delta_{il}\delta_{jk}1_N$\,, which implies that
\begin{eqnarray}\label{eq4444}
v_{ij}v_{kl}=0_N\,, \quad i\neq j\neq k\neq l\,.
\end{eqnarray}
Similarly, if $b\in B_6$\,, i.e., $b=a_i^4$\,, for some $1\leq i\leq d$\,, then the same argument applies, and we obtain $0_N=w_b=v_{ii}^2-\delta_{ii}^21_N$\,, i.e.,
\begin{eqnarray}\label{eq44444}
v_{ii}^2=1_N\,, \quad 1\leq i\leq d\,.
\end{eqnarray}
On the other hand, if $b\in B_2$\,, then by (\ref{eq8889990}) it follows that $b=a_i a_j a_k^2=a_i a_k^2 a_j=a_i a_j a_k^2$\,, for some $1\leq i\neq j\neq k\neq i\leq d$\,, and therefore
$w_b=v_{ij}v_{kk}^2+v_{ik}v_{kj}+v_{kk}^2v_{ij}-(\delta_{ik}\delta_{jk}+\delta_{ij}\delta_{kk}+\delta_{kj}\delta_{ki})1_N$\,. Hence
\begin{eqnarray}\label{eq44445}
v_{ij}v_{kk}^2+v_{ik}v_{kj}+v_{kk}^2v_{ij}=0_N\,, \quad 1\leq i\neq j\neq k\neq i\leq d\,.
\end{eqnarray}
Similarly, using $w_b=0_N$ for all $b\in B_m$\,, where $3\leq m\leq 5$\,, we obtain that the following relations hold for all $1\leq i\neq j\leq d$:
\begin{eqnarray}
v_{ii}v_{ij}+v_{ij}v_{ii}&=&0_N\label{eq444455}\\
v_{ij}v_{jj}+v_{jj}v_{ij}&=&0_N\label{eq444456}\\
v_{ii}v_{jj}+v_{ij}v_{ji}+v_{jj}v_{ii}+v_{ji}v_{ij}&=&2 1_N\,.\label{eq444457}
\end{eqnarray}
Now, recall that $v_{ij}^*=v_{ji}$\,, $1\leq i, j\leq d$, so by (\ref{eq44444}), we deduce that $\{v_{ii}\,, 1\leq i\leq d\}$ is a set of self-adjoint unitaries. Thus, by (\ref{eq444457}) we have \[ \|v_{ij} v_{ij}^*+v_{ij}^*v_{ij}\|=\|2 1_N-v_{ii} v_{jj}-v_{jj}v_{ii}\|\leq 4\,, \quad 1\leq i\neq j\leq d\,, \]
which implies that $\|v_{ij}\|\leq 2$\,, for all $1\leq i\neq j\leq d$\,.
Now, for every $1\leq j\leq d$\, set $p_j\colon =\bigvee_{i\neq j} s(v_{ij}^*v_{ij})$\,,
where $s(v_{ij}^*v_{ij})$ denotes the support projection of $v_{ij}^*v_{ij}$\,.
By (\ref{eq4444}) it follows that $p_j$ and $p_k$ are orthogonal projections, whenever $1\leq j\neq k\leq d$\,, and hence
\begin{eqnarray}\label{eq444555666}
\sum\limits_{j=1}^d \tau_N(p_j)\leq \tau_N(1)=1\,.
\end{eqnarray}
On the other hand, by (\ref{eq345457}), $\tau_N(v_{ij}^*v_{ij})=1$\,, for all $1\leq i, j\leq d$\,. Moreover, for $i\neq j$\,,
\[ v_{ij}^*v_{ij}\leq \|v_{ij}^*v_{ij}\|p_j\leq 4 p_j\,. \]
Thus $\tau_N(p_j)\geq ({\tau_N}(v_{ij}^*v_{ij}))/4={{1}/{4}}$\,, for all $1\leq i\neq j\leq d$\,.
This implies that $\sum_{j=1}^d \tau_N(p_j)\geq {d}/{4}$\,,
and since $d\geq 5$\,, this contradicts (\ref{eq444555666}). The proof is complete.
\end{proof}

The condition $d\geq 5$ is essential in the statement of Lemma \ref{lem6767} above, as it can be seen from the following remark.

\begin{rem}\rm
Assume that $T: M_n(\mathbb{C})\rightarrow M_n(\mathbb{C})$ is of the form (\ref{eq22233332222}), where $a_1\,, \ldots\,, a_d\in M_n(\mathbb{C})$ are self-adjoint with $\sum_{i=1}^d a_i^2=1_n$\,. If $d\leq 4$\,, then $T^2$ is factorizable.
\end{rem}

\begin{proof} We can assume without loss of generality that $d=4$ (otherwise add zero-terms). Set
\[ u\colon =\sum\limits_{i, j=1}^4 a_i a_j\otimes (2 e_{ij}-\delta_{ij}1_4)\,, \] where $(e_{ij})_{1\leq i, j\leq 4}$ are the standard matrix units in $M_4(\mathbb{C})$\,.
Then $u=u^*\in M_n(\mathbb{C})\otimes M_4(\mathbb{C})=M_{4n}(\mathbb{C})$\,. We first show that $u$ is a unitary. We have
\begin{eqnarray*}
u^*u \,\, = \,\, uu^* \,\, =u^2&=&\sum\limits_{i, j, k, l=1}^4 a_i a_j a_k a_l \otimes (2e_{ij}-\delta_{ij}1_4)(2 e_{kl}-\delta_{kl}1_4)\\
&=&\sum\limits_{i, j, k, l=1}^4 a_i a_j a_k a_l \otimes (4\delta_{jk}e_{il}-2\delta_{ij}e_{kl}-2\delta_{kl}e_{ij}+\delta_{ij}\delta_{kl}1_4)\\
&=& 4s_1-2s_2-2s_3+s_4\,,
\end{eqnarray*}
where $s_1\colon =\sum_{i,j,l=1}^4 a_i a_j^2 a_l\otimes e_{il}=\sum_{i,l=1}^4 a_i a_l\otimes e_{il}$\,, $s_2\colon =\sum_{i, k, l=1}^4 a_i^2a_ka_l\otimes e_{kl}=\sum_{k, l=1}^4 a_ka_l\otimes e_{kl}$\,, $s_3\colon =\sum_{i, j, k=1}^4 a_ia_ja_k^2\otimes e_{ij}=\sum_{i, j=1}^4 a_ia_j\otimes e_{ij}$ and $s_4\colon =\sum_{i,j=1}^4 a_i^2a_j^2\otimes 1_4=1_{4n}$\,. We have repeatedly used the fact that $\sum_{i=1}^d a_i^2=1_n$\,. Hence $s_1=s_2=s_3$\,, and therefore
\[ u^*u=uu^*=s_4=1_{4n}\,. \]
Next we prove that $T^2$ is factorizable by showing that
\[ \mathbb{E}_{M_n(\mathbb{C})\otimes {1_4}}(u^*(x\otimes 1_4)u)=T^2x\,, \quad x\in M_n(\mathbb{C})\,. \]
Since $u=u^*$\,, the left hand side above becomes
\begin{eqnarray*}
 \mathbb{E}_{M_n(\mathbb{C})\otimes {1_4}}(u^*(x\otimes 1_4)u)&=&\mathbb{E}_{M_n(\mathbb{C})\otimes {1_4}}\left(\sum\limits_{i, j, k, l=1}^4 a_i a_j\otimes (2 e_{ij}-\delta_{ij}1_4)(x\otimes 1_4)(a_k a_l\otimes (2e_{kl}-\delta_{kl}1_4))\right)\\
&=& \sum\limits_{i, j, k, l=1}^4 a_i a_jxa_ka_l \tau_4((2e_{ij}-\delta_{ij}1)(2e_{kl}-\delta_{kl}I))\\
&=& \sum\limits_{i, j=1}^4 a_ia_jxa_ja_i\\&= &T^2x\,,
\end{eqnarray*}
 wherein we have used that fact that $\tau_4((2e_{ij}-\delta_{ij}1_4)(2e_{kl}-\delta_{kl}1_4))=({4}/{4}) \delta_{il}\delta_{jk}-({2}/{4})\delta_{ij}\delta_{kl}-({2}/{4})\delta_{ij}\delta_{kl}+\delta_{ij}\delta_{kl}=\delta_{il}\delta_{jk}$\,.
The proof is complete.
\end{proof}

\begin{lemma}\label{lem77887788}
Let $d\geq 5$\,, $b_1\,, \ldots\,, b_d$ be self-adjoint matrices in $M_m(\mathbb{C})$\,, and $u_1\,, \ldots\,, u_d$ be self-adjoint unitary matrices in $M_r(\mathbb{C})$\,, where $m$ and $r$ are positive integers. Assume that
\begin{enumerate}
\item [$(a)$] $\sum_{i=1}^d b_i^2=1_m$\,.
\item [$(b)$] $b_ib_j=b_jb_i$\,, for all $1\leq i, j\leq d$\,.
\item [$(c)$] $b_ib_jb_kb_l\neq 0_m$\,, for all $1\leq i, j, k, l\leq d$\,.
\item [$(d)$] For every $1\leq i\neq j\leq d$\,, the set $\{b_ib_jb_k^2: 1\leq k\leq d\}$ is linearly independent in $M_m(\mathbb{C})$\,.
\item [$(e)$] The set $\{b_i^2b_j^2: 1\leq i< j\leq d\}$ is linearly independent in $M_m(\mathbb{C})$\,.
\item [$(f)$] The set $\{1_r\}\cup \{u_iu_j: 1\leq i\neq j\leq d\}\cup \{u_iu_ju_ku_l: 1\leq i\neq j\neq k\neq l\leq d\}$ is linearly independent in $M_r(\mathbb{C})$\,.
\end{enumerate}
Then $a_i\colon= b_i\otimes u_i$\,, $1\leq i\leq d$ are self-adjoint matrices in $M_{mr}(\mathbb{C})=M_m(\mathbb{C})\otimes M_r(\mathbb{C})$ which satisfy $\sum_{i=1}^d a_i^2=1_{mr}$\,, as well as the conditions $(i)-(iv)$ in Lemma \ref{lem6767} with $n=mr$\,.
\end{lemma}

\begin{proof}
Note first that conditions $(a)$ and $(b)$\,, together with the fact that $u_i^2=1_r$\,, for all $1\leq i\leq d$\,, imply that $\sum_{j=1}^d a_j^2=1_{mr}$ and $a_i^2a_j=a_ja_i^2$\,, $1\leq i, j\leq d$\,, i.e., condition $(i)$ in Lemma \ref{lem6767} is satisfied. Further,
the set $A\colon=\{a_ia_j: 1\leq i, j\leq d\}$ is equal to
\[ \{b_i^2\otimes 1_r: 1\leq i\leq d\}\cup \{b_ib_j\otimes u_iu_j: 1\leq i\neq j\leq d\}\,. \]
By $(e)$\,, the set $\{1_r\}\cup \{u_iu_j:1\leq i\neq j\leq d\}$ is linearly independent. Hence $A$ is linearly independent if and only if $b_ib_j\neq 0_m$\,, whenever $1\leq i\neq j\leq d$\,. The linear independence of $b_1^2\,, \ldots\,, b_d^2$ follows from $(d)$\,, and by $(c)$ we get that $b_ib_j\neq 0_m$\,, for all $1\leq i, j\leq d$\,. This proves condition $(ii)$ in Lemma \ref{lem6767}.

Next, consider the set $B\colon =B_1\cup \ldots \cup B_6$\,, where $B_1\,, \ldots \,,B_6$ are defined as in $(iii)$ in the above mentioned lemma. Since $u_i^2=1_r$\,, for all $1\leq i\leq d$\,, the sets $B_1\,, \ldots\,, B_6$ can be rewritten as: $B_1=\{b_ib_jb_kb_l\otimes u_iu_ju_ku_l: 1\leq i\neq j\neq k\neq l\leq d\}$\,, $B_2=\{b_ib_jb_k^2\otimes u_iu_j: 1\leq i\neq j\neq k\neq i\leq d\}$\,, $B_3=\{b_i^3b_j\otimes u_iu_j:1\leq i\neq j\leq d\}$\,, $B_4=\{b_ib_j^3\otimes u_iu_j: 1\leq i\neq j\leq d\}$\,, $B_5=\{b_i^2b_j^2\otimes 1_l: 1\leq i< j\leq d\}$ and $B_6=\{b_i^4\otimes 1_l: 1\leq i\leq d\}$\,. By $(e)$\,, $B$ is a linearly independent set if and only if the following three conditions hold:
\begin{enumerate}
\item [$(1)$] $b_i b_j b_k b_l\neq 0_m$\,, whenever $1\leq i\neq j\neq k\neq l\leq d$\,.
\item [$(2)$] For every $1\leq i\neq j\leq d$\,, the set $\{b_ib_jb_k^2: 1\leq k\leq d, k\neq i\,, k\neq j\}\cup \{b_i^3b_j: 1\leq i, j\leq d\}\cup \{b_ib_j^3: 1\leq i, j\leq d\}$ is linearly independent.
\item [$(3)$] The set $\{b_i^2b_j^2: 1\leq i< j\leq d\}\cup \{b_i^4: 1\leq i\leq d\}$ is linearly independent.
\end{enumerate}
Clearly, $(c)$ implies $(1)$\,, $(e)$ implies $(3)$\,, and by $(b)$\,, condition $(d)$ implies $(2)$\,. Hence $(iii)$ in Lemma \ref{lem6767} holds, and since $d\geq 5$\,, condition $(iv)$ holds, as well, thus completing the proof.
\end{proof}

{\bf Proof of Theorem \ref{th5656}}: It remains to be proved that for $d\geq 5$\,, there exist positive integers $m, r$ and matrices $b_1\,, \ldots\,, b_d\in M_m(\mathbb{C})$\,, $u_1\,, \ldots\,, u_d\in M_r(\mathbb{C})$ satisfying the hypotheses of Lemma \ref{lem77887788}.

Let $S^{d-1}=S(\mathbb{R}^d)$ denote the unit sphere in $\mathbb{R}^d$\,, and let $\phi_1\,, \ldots\,, \phi_d:S^{d-1}\rightarrow \mathbb{R}$ be the coordinate functions. It is not difficult to check that these functions in $C(S^{d-1})$ satisfy conditions $(a)-(e)$ in Lemma \ref{lem77887788}. Conditions $(a)-(c)$ are, indeed, obvious. To prove $(e)$\,, note that $\phi_d^2=1-\phi_1^2-\ldots -\phi_{d-1}^2$\,. Hence, the linear independence of the set $\{\phi_i^2\phi_j^2:1\leq i< j\leq d\}$ is equivalent to the linear independence of the set of polynomials
\[ {\mathcal P}:=\{x_i^2x_j^2: 1\leq i\leq j\leq d-1\}\cup \{x_i^2: 1\leq i\leq d-1\}\cup \{1\} \]
in $C(B(\mathbb{R}^{d-1}))$\,, where $B(\mathbb{R}^{d-1})$ is the closed unit ball in $\mathbb{R}^{d-1}$\,. But ${\mathcal P}$ is clearly a linearly independent set, because if a polynomial in $\mathbb{R}^{d-1}$ vanishes in a neighborhood of 0, then all its coefficients are 0. This shows that $\phi_1\,, \ldots\,, \phi_d$ satisfy $(e)$\,. The same method gives that $\{\phi_1^2\,, \ldots\,, \phi_d^2\}$ is a linearly independent set, and since for $1\leq i\neq j\leq d$\,, the set $\{x\in S^{d-1}: \phi_i(x)\phi_j(x)\neq 0\}$ is dense in $S^{d-1}$\,, it follows that also condition $(d)$ holds for $\phi_1\,, \ldots\,, \phi_d$\,.

Next we show that $(a)-(d)$ hold for the restriction of $(\phi_1\,, \ldots\,, \phi_d)$ to some finite subset of $S^{d-1}$\,. For this, assume that $(e)$ fails for the restriction of $(\phi_1\,, \ldots\,, \phi_d)$ to any finite subset $F$ of $S^{d-1}$\,. Then, for each such $F$\,, we can find coefficients $c_{ij}^F$\,, $1\leq i\leq j\leq d$\,, not all equal to zero, such that
\[ \sum\limits_{1\leq i\leq j\leq d} c_{ij}^F \phi_i^2(x)\phi_j^2(x)=0\,, \quad x\in F\,. \]
Moreover, we can assume that $\sum_{1\leq i\leq j\leq d} |c_{ij}^F|^2=1$\,. Take now a weak$^*$-limit point $c=(c_{ij})_{1\leq i\leq j\leq d}$ of the net $((c_{ij}^F)_{1\leq i\leq j\leq d})_F$\,, where the finite subsets $F\subseteq S^{d-1}$ are ordered by inclusion. Then \[ \sum\limits_{1\leq i\leq j\leq d} c_{ij}\phi_i^2(x)\phi_j^2(x)=0\,, \quad x\in S^{d-1}\,, \]
and not all coefficients $c_{ij}$ above vanish. This contradicts the fact that $\phi_1\,, \ldots\,, \phi_d$ satisfy $(e)$\,. Using this type of argument, it is easy to see that one can choose a finite subset $F$ of $S^{d-1}$ such that not only $(e)$\,, but also $(d)$ and $(c)$ hold for the restrictions of $\phi_1\,, \ldots\,, \phi_d$ to $F$\,. Of course, conditions $(a)$ and $(b)$ also hold for these restrictions. Set now $m:=|F|$\,, and let $F=\{p_1\,, \ldots\,, p_m\}$\,. Then the diagonal matrices \[ b_i\colon =\text{diag}\{\phi_i(p_1)\,, \ldots\,, \phi_i(p_m)\}\,, \quad 1\leq i\leq d \]
in $M_m(\mathbb{C})$ satisfy $(a)-(d)$\,.

It remains to be proved that we can find self-adjoint unitaries $u_1\,, \ldots\,, u_d$ in some matrix algebra $M_r(\mathbb{C})$ satisfying $(f)$\,. For this, consider the free product group $G\colon =\mathbb{Z}_2\ast \ldots \ast\mathbb{Z}_2$ ($d$ copies), and let $g_1\,, \ldots\,, g_d$ be its generators. Then $g_i^2=1$\,, for all $1\leq i\leq d$ and $g_{i_1}g_{i_2}\ldots g_{i_s}\neq 0$\,, whenever $s$ is a positive integer and $i_1\neq i_2\neq \ldots \neq i_s$\,. Since $G$ is residually finite (cf. \cite{Gru}), by passing to a quotient of $G$ we can find a finite group $\Gamma$ generated by $\gamma_1\,, \ldots\,, \gamma_d$ such that $\gamma_i^2=1$\,, for all $1\leq i\leq d$\,, and $\gamma_{i_1}\gamma_{i_2}\ldots \gamma_{i_s}\neq 0$\,, whenever $1\leq s\leq d$ and $i_1\neq i_2\neq \ldots \neq i_s$\,. This implies that the group elements listed in the set $\{1\}\cup \{\gamma_i\gamma_j: 1\leq i\neq j\leq d\}\cup \{\gamma_{i}\gamma_j\gamma_k\gamma_l: 1\leq i\neq j\neq k\neq l\leq d\}$ are all distinct. Set now $r\colon =|\Gamma|$ and let $u_1\,, \ldots \,, u_d$ be the ranges of $\gamma_1\,, \ldots\,, \gamma_d$ by the left regular representation $\lambda_\Gamma$ of $\Gamma$ in ${\mathcal B}(l^2(\Gamma))\simeq M_r(\mathbb{C})$\,. Then $u_1\,, \ldots\,, u_d$ are self-adjoint unitaries. Moreover, the set $\{\lambda_{\Gamma}(\gamma): \gamma\in \Gamma\}$ is linearly independent, because $\lambda_\Gamma(\gamma)\delta_e=\delta_{
\gamma}$\,, where $\delta_{\gamma}\in l^2(\Gamma)$ is defined by $\delta_{\gamma}({\gamma}^\prime)= 1$\,, if $\gamma^{\prime}=\gamma$ and $\delta_{\gamma}({\gamma}^\prime)= 0$\,, else.
Hence $u_1\,, \ldots \,, u_d$ satisfy $(f)$ and the proof is complete.\qed

\begin{rem}\label{remat78}\rm
The above proof does not provide explicit numbers $m$ and $r$ for which $b_1\,, \ldots \,, b_d$ and $u_1\,, \ldots\, u_d$ can be realized, but it is easy to find lower bounds. For $d=5$\,, the set in $(e)$ has 15 linearly independent elements, and since the $b_i$'s can be simultaneously diagonalized, it follows that $m\geq 15$\,. Also, for $d=5$\,, the set in $(f)$ has $1+5\times 4+5\times 4^3=341$ linearly independent elements in $M_r(\mathbb{C})$\,. Hence $r^2\geq 341$\,, which implies that $r\geq 19$. Altogether, we conclude that $n\colon =mr\geq 15\times 19=285$\,.
\end{rem}

We will end this section by giving a characterization of those $\tau_M$-Markov maps $S:M\rightarrow M$ which admit a Rota dilation.

\begin{theorem}\label{th666666}
Let $M$ be a finite von Neumann algebra with normal, faithful, tracial state $\tau_M$\,, and let $S\colon M\rightarrow M$ be a linear operator.
Then the following statements are equivalent:
\begin{enumerate}
\item [$(i)$] $S$ satisfies the Rota dilation property.
\item [$(ii)$] $S=T^*T$\,, for some factorizable $(\tau_M\,, \tau_N)$-Markov map $T:M\rightarrow N$ taking values in a von Neumann algebra $N$ with a normal, faithful, tracial state $\tau_N$\,.
\end{enumerate}
\end{theorem}

\begin{proof}
The implication $(i)\Rightarrow (ii)$ follows immediately from Remark \ref{rem102} (see (\ref{eq345}) therein). We now prove that $(ii)\Rightarrow (i)$.
Suppose that there exists a factorizable $(\tau_M\,, \tau_N)$-Markov map $T:M\rightarrow N$\,, where $N$ is a von Neumann algebra with a normal, faithful, tracial state $\tau_N$ such that
\[ S=T^*T\,. \]
Since $T$ is factorizable, it follows by Remark \ref{rem787877777} $(a)$ that there exists a finite von Neumann algebra $P$ with a normal, faithful tracial state $\tau_P$ such that
$T=\beta^*\circ \alpha$\,,
where $\alpha:M\rightarrow P$ and $\beta:N\rightarrow P$ are unital $*$-monomorphisms satisfying $\tau_M=\tau_P\circ \alpha$ and $\tau_N=\tau_P\circ \beta$\,. Consider now the von Neumann algebras $M\oplus N$ and $P\oplus P$\,, equipped with the normal, faithful, tracial states defined by $\tau_{M\oplus N}\colon =(\tau_M\oplus \tau_N)/2$ and $\tau_{P\oplus P}\colon =(\tau_P\oplus \tau_P)/2$\,, respectively. Further, define an operator $\tilde{T}$ on $M\oplus N$ by
\[ \tilde{T}(x, y)\colon =(T^*y, Tx)\,, \quad x\in M\,, y\in N\,. \]
Then $\tilde{T}$ is a $\tau_{M\oplus N}$-Markov map on $M\oplus N$ and ${\tilde{T}}^*=\tilde{T}$\,. Moreover, $\tilde{T}$ is factorizable, since $\tilde{T}=\delta^*\circ \gamma$\,, where $\delta\,, \gamma:M\oplus N\rightarrow P\oplus P$ are the $*$-monomorphisms given by
$\gamma(x, y)\colon =(\alpha(x), \beta(y))$\,, respectively, $\delta(x, y)\colon =(\beta(y), \alpha(x))$\,, $x\in M$\,, $y\in N$\,,
and $\tau_{M\oplus N}=(\tau_{P\oplus P})\circ \gamma=(\tau_{P\oplus P})\circ \delta$\,.

Hence, by Theorem \ref{th6767}, $\tilde{T}^2$ has a Rota dilation, i.e., there exists a finite von Neumann algebra $Q$ with a normal, faithful, tracial state $\tau_Q$, a unital $*$-monomorphism $i\colon M\oplus N\rightarrow Q$ for which $\tau_{M\oplus N}=\tau_Q\circ i$\,, and a decreasing sequence $(Q_n)_{n\geq 1}$ of von Neumann subalgebras of $Q$ such that
\[ \tilde{T}^{2n}=i^*\circ \mathbb{E}_{Q_n}\circ i\,, \quad n\geq 1\,, \]
where $\mathbb{E}_{Q_n}$ is the unique $\tau_Q$-preserving normal conditional expectation of $Q$ onto $Q_n$\,. Note that
\[ \tilde{T}^{2n}(x, y)=((T^*T)^nx\,, (TT^*)^ny)\,, \quad (x, y)\in {M\oplus N}\,. \]
In particular, $\tilde{T}^{2n}(1_M, 0_N)=(1_M\,, 0_N)$\,. Set $e\colon =i((1_M\,, 0_N))$\,.
Then $e$ is a projection in $Q$\,. We will show next that $e\in Q_n$\,, for all $n\geq 1$\,. For simplicity of notation, set $z\colon =(1_M, 0_N)$ and $w\colon =1_{M\oplus N}-z$\,. For all $n\geq 1$\,, $w^*\tilde{T}^{2n}(z)=w^*z=0_{M\oplus N}$\,, and therefore
\begin{eqnarray*}
0&=& \langle \tilde{T}^{2n} (z)\,, w\rangle_{L^2(M\oplus N)}\\
&=& \langle (i^*\circ {\mathbb{E}}_{Q_n} \circ i)(z)\,, w\rangle_{L^2(M\oplus N)}\\
&=& \langle {\mathbb{E}}_{Q_n} (i(z))\,, i(w)\rangle_{L^2(Q)}\\
&=&\tau_Q((1_Q-e){\mathbb{E}}_{Q_n}(e))\\
&=& \tau_Q((1_Q-e){\mathbb{E}}_{Q_n}(1_Q-e))\,.
\end{eqnarray*}
Since $\tau_Q$ is faithful and ${\mathbb{E}}_{Q_n}(e)\geq 0$\,, it follows that ${\mathbb{E}}_{Q_n}(e)\in e Q e$\,. Similarly, we obtain that ${\mathbb{E}}_{Q_n}(1-e)\in (1_Q-e)Q(1_Q-e)$\,. Since ${\mathbb{E}}_{Q_n}(e)-e=(1_Q-e)-{\mathbb{E}}_{Q_n}(1_Q-e)$\,, we deduce that \[ {\mathbb{E}}_{Q_n}(e)-e\in eQe\cap (1_Q-e)Q(1_Q-e)=\{0_Q\}\,, \quad n\geq 1\,, \]
which proves the claim.
Further, note that $\tau_Q(e)=\tau_{M\oplus N}((1_N\,, 0_N))=1/2$\,.
Set $R\colon =eQe$\,, $\tau_R\colon =2{(\tau_Q)_{\vert_R}}$\,, and for $x\in M$\,, let $j(x)\colon = i(x, 0_N)$\,. Then it is easy to check that $R$ is a von Neumann algebra with normal, faithful tracial state $\tau_R$ and that the map $j:M\rightarrow R$ above defined is a unital $*$-monomorphism for which $\tau_M=j\circ \tau_R$\,. Moreover, for all $n\geq 1$\,,
\[ (T^*T)^n=j^*\circ {\mathbb{E}}_{R_n}\circ j\,, \quad n\geq 1\,, \]
where $R_n\colon =eQ_n e$\,, $n\geq 1$\,, form a decreasing sequence of von Neumann subalgebras of $R$\,, and ${\mathbb{E}}_{R_n}\colon =({\mathbb{E}}_{Q_n})_{\vert_{R}}$ is the unique $\tau_R$-preserving conditional expectation of $R$ onto $R_n$\,. It follows by the definition that $S=T^*T$ has a Rota dilation.
\end{proof}

Note that from the proof of Theorem \ref{th666666} it follows right-away that in order for a linear map $T\colon M\rightarrow M$ to satisfy the Rota dilation property, it suffices that it satisfies the conditions set forth in Definition \ref{defrotadil} for $m=1$, only.

\section{On the asymptotic quantum Birkhoff conjecture}

In 1946 G. Birkhoff \cite{Bir} proved that every doubly stochastic matrix is a convex combination of permutation matrices.
Note that if one considers the abelian von Neumann algebra $D\colon =l^\infty(\{1, 2, \ldots\,, n\})$ with trace given by $\tau(\{i\})=1/n$\,, $1\leq i\leq n$\,, then the positive unital trace-preserving maps on $D$ are those linear operators on $D$ which are given by doubly stochastic $n\times n$ matrices. Since every automorphism of $D$ is given by a permutation of $\{1\,, 2\,, \ldots\,, n\}$, this led naturally to the question whether Birkhoff's classical result extends to the quantum setting. The statement that
every completely positive, unital trace-preserving map $T\colon (M_n(\mathbb{C}), \tau_n)\rightarrow (M_n(\mathbb{C}), \tau_n)$ lies in $\text{conv}(\text{Aut}(M_n(\mathbb{C})))$ turned out to be false for $n\geq 3$\,. For the case $n\geq 4$, this was shown by K\"ummerer and Maasen (cf. \cite{KM}), while the case $n=3$ was settled by K\"ummerer in \cite{Ku1} (see Remark \ref{rem8080}). In \cite{LS}, Landau and Streater gave a more elementary proof of K\"ummerer and Maasen's result, and also constructed another counterexample to the quantum Birkhoff conjecture in the case $n=3$\,.

Recently, V. Paulsen brought to our attention the following asymptotic version of the quantum Birkhoff conjecture, listed as Problem 30 on R. Werner's web page of open problems in quantum information theory (see \cite{We}):

{\bf The asymptotic quantum Birkhoff conjecture}: Let $n\geq 1$\,. If $T\colon M_n(\mathbb{C})\rightarrow M_n(\mathbb{C})$ is a $\tau_n$-Markov map, then $T$ satisfies the following {\em asymptotic quantum Birkhoff property}:
\begin{eqnarray}\label{eq5}\lim\limits_{k\rightarrow \infty} d_{\text{cb}}\bigg({\textstyle{\bigotimes\limits_{i=1}^k}} \,T\,, \text{conv}(\text{Aut}({\textstyle{\bigotimes\limits_{i=1}^k}} M_n(\mathbb{C})))\bigg)=0\,. \end{eqnarray}

As mentioned in the introduction, this conjecture originates in joint work of A. Winter, J. A. Smolin and
F. Verstraete. The main results obtained in \cite{SVW} motivated its formulation. We would like to thank M.-B. Ruskai for kindly providing us with the Report of the workshop on Operator structures in quantum information theory that took place at BIRS, February 11-16, 2007, where A. Winter discussed the conjecture, as well as for pointing out related very recent work of C. Mendl and M. Wolf (cf. \cite{MWo}).

Using the existence of non-factorizable Markov maps, we prove the following:

\begin{theorem}\label{thasybirkh}
For every $n\geq 3$\,, there exist $\tau_n$-Markov maps on $M_n(\mathbb{C})$ which do not satisfy the asymptotic quantum Birkhoff property (\ref{eq5}).
\end{theorem}

\begin{proof}
We will show that any non-factorizable $\tau_n$-Markov map on $M_n(\mathbb{C})$ does not satisfy (\ref{eq5}). Such maps do exist for every $n\geq 3$\,, as it was shown in Section 3.

The key point in our argument is to prove that any $\tau_n$-Markov map $T\colon M_n(\mathbb{C})\rightarrow M_n(\mathbb{C})$\,, $n\geq 3$\,, satisfies the following inequality:
\begin{equation}\label{eq334433443}d_{\text{cb}}\left({\textstyle{\bigotimes\limits_{i=1}^k}} \,T\,, {\mathcal F}{\mathcal M}\left({\textstyle{\bigotimes\limits_{i=1}^k}} M_n(\mathbb{C})\right)\right)\geq  d_{\text{cb}}(T, \,\,{\mathcal F}{\mathcal M}(M_n(\mathbb{C})))\,, \quad k\geq 1\,. \end{equation}
Then, since $\text{conv}(\text{Aut}({\textstyle{\bigotimes_{i=1}^k}} M_n(\mathbb{C})))\subset {\mathcal F}{\mathcal M}\left({\textstyle{\bigotimes_{i=1}^k}} M_n(\mathbb{C})\right)$\,, for all $k\geq 1$\,, the desired conclusion will follow immediately, using the fact that the set of factorizable maps on $M_n(\mathbb{C})$ is closed in the norm-topology, cf. Remark \ref{rem787877777} $(b)$. (Note that in our concrete finite-dimensional setting, this latter fact can also be obtained directly from Theorem \ref{th1} using a simple ultraproduct argument.)

Now, in order to prove (\ref{eq334433443}), we show that given $m, l\geq 3$\,, then for any $\tau_m$-Markov map $T$ on $M_m(\mathbb{C})$ and any $\tau_l$-Markov map $S$ on $M_l(\mathbb{C})$\,, we have \begin{equation}\label{eq3334443000000000} d_{\text{cb}}(T\otimes S\,, {\mathcal F}{\mathcal M}(M_m(\mathbb{C})\otimes M_l(\mathbb{C})))\geq d_{\text{cb}}(T, \,\,{\mathcal F}{\mathcal M}(M_m(\mathbb{C})))\,. \end{equation}
Let $\iota\colon M_m(\mathbb{C})\rightarrow M_m(\mathbb{C})\otimes M_l(\mathbb{C})$ be defined by $\iota(x)\colon = x\otimes 1_l$\,, for all $x\in M_m(\mathbb{C})$\,. Then its adjoint map $\iota^*\colon M_m(\mathbb{C})\otimes M_l(\mathbb{C})\rightarrow M_m(\mathbb{C})$ is given by
$\iota^*(z)=(1_m\otimes \tau_l)(z)$\,, for all $z\in M_m(\mathbb{C})\otimes M_l(\mathbb{C})$\,.
It is easily checked that $\iota^*(T\otimes S)\iota=T$\,. Since $\|\iota\|_{\text{cb}}=\|\iota^*\|_{\text{cb}}=1$\,, we then obtain
\begin{eqnarray}\label{eq67} d_{\text{cb}}(T\otimes S\,, {\mathcal F}{\mathcal M}(M_m(\mathbb{C})\otimes M_l(\mathbb{C})))&\geq & d_{\text{cb}}(T, \,\,\iota^*{\mathcal F}{\mathcal M}(M_m(\mathbb{C})\otimes M_l(\mathbb{C}))) \iota)\,. \end{eqnarray}
By the permanence properties of factorizability,
$\iota^* {\mathcal F}{\mathcal M}(M_m(\mathbb{C})\otimes M_m(\mathbb{C})) \iota\subset {\mathcal F}{\mathcal M}(M_m(\mathbb{C}))$\,.
Together with (\ref{eq67}), this completes the proof of (\ref{eq3334443000000000}), which, in turn, yields (\ref{eq334433443}).
\end{proof}

It is now a natural question whether every factorizable $\tau_n$-Markov map on $M_n(\mathbb{C})$ does satisfy the asymptotic quantum Birkhoff property, for all $n\geq 3$\,. It turns out that this question has a tight connection to Connes' embedding problem, which is known to be equivalent to a number of different fundamental
problems in operator algebras (for references, see, e.g., Ozawa's excellent survey paper \cite{Oz}).

\begin{theorem}
If for any $n\geq 3$\,, every factorizable $\tau_n$ -Markov map on $M_n(\mathbb{C})$ satisfies the asymptotic quantum Birkhoff property, then Connes' embedding problem has a positive answer.
\end{theorem}

\begin{proof}
Assume by contradiction that Connes's embedding problem has a negative answer. Then, by Dykema and Juschenko's results from \cite{DyJ} (as explained in Remark \ref{rem6666777}), there exists a positive integer $n$ such that ${{\mathcal G}_n\setminus {{\mathcal F}_n}}\neq \emptyset$\,. Choose $B\in {{\mathcal G}_n\setminus {{\mathcal F}_n}}$\,. It follows that the associated Schur multiplier $T_B$ is factorizable.

We will prove that $T_B$ does not satisfy the asymptotic quantum Birkhoff property. Suppose by contradiction that $T_B$ does satisfy (\ref{eq5})\,. For every positive integer $k$\,, let $\iota_k\colon M_n(\mathbb{C})\rightarrow M_{n^k}(\mathbb{C})$ be the map defined by
$\iota_k(x)\colon = x\otimes 1_n\otimes \ldots \otimes 1_n$\,, for all $x\in M_n(\mathbb{C})$\,.
We deduce that
\begin{equation*}
\lim\limits_{k\rightarrow \infty } d_{\text{cb}}(T, \,\, \text{conv}(\iota_k^*\circ \text{Aut}(M_{n^k}(\mathbb{C})) \circ \iota_k))=0\,.
\end{equation*}
For $k\geq 1$\,, choose operators $T_k\in \text{conv}(\iota_k^*\circ \text{Aut}(M_{n^k}(\mathbb{C})) \circ \iota_k)$ such that
\begin{equation}\label{eq785600}
\lim\limits_{k\rightarrow \infty}\|T-T_k\|_{\text{cb}}=0\,.
\end{equation}
 Each $T_k$ is of the form $T_k=\sum_{i=1}^{s_k} c_i^{(k)} \iota_k \circ \text{ad}(u_i^{(k)}) \circ \iota_k$\,, for some positive integer $s_k$\,, unitaries $u_i^{(k)}\in {\mathcal U}(n^k)$ and positive numbers $c_i^{(k)}$\,, $1\leq i\leq s_k$\,, with $\sum_{i=1}^{s_k} c_i^{(k)}=1$\,.

 Set $u_k\colon = (u_1^{(k)}\,, \ldots\,, u_{s_k}^{(k)})$\,. Then $u_k\in \bigoplus_{i=1}^{s_k} M_{n^k}(\mathbb{C})\colon = A_k$\,. Equip $A_k$ with the trace given by $\tau((a_1\,, \ldots\,, a_{s_k}))\colon = \sum_{i=1}^{s_k} c_i^{(k)} \tau_{n^k}(a_i)$\,, for all $(a_1\,, \ldots\,, a_{s_k})\in A_k$\,. Finally, define $j_k\colon M_n(\mathbb{C})\rightarrow A_k$ by $j_k(x)\colon = (\iota_k(x)\,, \ldots\,, \iota_k(x))$ ($s_k$ terms)\,, for all $x\in M_n(\mathbb{C})$\,. It can be checked that the adjoint $j_k^*$ of $j_k$ is given by $j_k^*((a_1\,, \ldots\,, a_{s_k}))=\iota_k^*(\sum_{i=1}^{s_k} c_i^{(k)}a_i)$\,, for all $(a_1\,, \ldots\,, a_{s_k})\in A_k$\,. Then $T_k$ can be rewritten as
\begin{equation}\label{eq5050998} T_k=j_k^* \circ \text{ad}(u_k) \circ j_k\,. \end{equation}
Since $A_k$ admits naturally a $\tau_k$-preserving embedding into the hyperfinite II$_1$-factor $R$\,, equipped with its trace $\tau_R$\,, we can replace in the above formula $A_k$ by $R$\,, and view $u_k$ as a unitary in $R$\,, respectively view $j_k$ as a unital embedding of $M_n(\mathbb{C})$ into $R$\,. By taking ultraproducts, and using (\ref{eq785600}) we obtain that \begin{equation*}
T=j^*\circ \text{ad}(u)\circ j\,,
\end{equation*}
where $u$ is a unitary in the ultrapower $R^\omega$ of $R$ and $j\colon M_n(\mathbb{C})\rightarrow R^\omega$ is a unital embedding. Using the identification
\begin{equation*}
R^\omega=j(M_n(\mathbb{C}))\otimes (j(M_n(\mathbb{C}))^\prime \cap R^\omega)\simeq M_n( j(M_n(\mathbb{C}))^\prime \cap R^\omega)\,,
\end{equation*}
we obtain from the proof of Lemma \ref{lem222222} applied to the factorizable Schur multiplier $T_B$ that $u=\text{diag}(u_1\,, \ldots\,, u_n)$\,, where $u_i\in j(M_n(\mathbb{C}))^\prime \cap R^\omega\subset R^\omega$\,, for all $1\leq i\leq n$, and that the $(k, l)$-th entry $B_{kl}$ of $B$ is given by $B_{kl}=\tau_{R^\omega}(u_k^*u_l)$\,, for all $1\leq k, l\leq n$\,. Here $\tau_{R^\omega}$ denotes the trace on $R^\omega$\,.

By a standard ultraproduct argument, for every $1\leq k\leq n$\,, we can find a sequence $(u_k^{(m)})_{m\geq 1}$ of unitaries in $R$ representing $u$\,, and we conclude  that
\begin{equation*}
\lim\limits_{m\rightarrow \infty} \tau_R\left(\left(u_k^{(m)}\right)^*u_l^{(m)}\right)=\tau_{R^\omega}(u_k^* u_l)\,, \quad 1\leq k, l\leq n\,.
\end{equation*}
Since ${{\mathcal F}_n}$ is a closed set, this shows that $B\in {{\mathcal F}_n}$\,, which contradicts the assumption on $B$\,. Therefore, the proof is complete.
\end{proof}

\section{On the best constant in the noncommutative little Grothendieck inequality}

Let $OH(I)$ denote Pisier's operator Hilbert space based on $l^2(I)$\,, for a given index set $I$\,. Further, let $A$ be a C$^*$-algebra and $T:A\rightarrow OH(I)$ a completely bounded map. Then, by the refinement of the second part of Corollary 3.4 of \cite{PiSh} obtained in \cite{HM}\,, there exist states $f_1$\,, $f_2$ on $A$ such that
\begin{equation}\label{eq33343}
\|Tx\|\leq \sqrt{2}\,\|T\|_{\text{cb}}{f_1(xx^*)}^{1/4}{f_2(x^*x)}^{1/4}\,, \quad x\in A\,.
\end{equation}

\begin{defi}\label{defbestconst}
For a completely bounded map $T:A\rightarrow OH(I)$ we denote by $C(T)$ the smallest constant $C> 0$ for which there exist states $f_1$\,, $f_2$ on $A$ such that
\begin{equation}\label{eq333435}
\|Tx\|\leq C {f_1(xx^*)}^{1/4}{f_2(x^*x)}^{1/4}\,, \quad x\in A\,.
\end{equation}
\end{defi}

The existence of a smallest constant $C(T)$ as above follows from a simple compactness argument using the fact that the set $Q(A)\colon= \{f\in A^*_{+}: \|f\|\leq 1\}$ is w$^*$-compact. From \cite{PiSh} we know that $\|T\|_{\text{cb}}\leq C(T)$\,. Hence, by (\ref{eq33343}), we infer that
\begin{equation}\label{eq8877887788787}
\|T\|_{\text{cb}}\leq C(T)\leq \sqrt{2}\,\|T\|_{\text{cb}}\,. \end{equation}
In the following we will prove that for suitable choices of $A$\,, $I$ and $T$\,, $C(T)> \|T\|_{\text{cb}}$\,, i.e., the constant $\sqrt{2}$ in (\ref{eq33343}) cannot be reduced to 1. It would be interesting to know what is the best constant $C_0$ in the noncommutative little Grothendieck inequality (\ref{eq33343})\,, i.e., what is the smallest constant $C_0$ for which $C(T)\leq C_0 \|T\|_{\text{cb}}$\,, for arbitrary choices of $A$\,, $I$ and $T$\,.

\begin{theorem}\label{th6000000000006}
There exist linear maps $T_1\colon M_3(\mathbb{C})\rightarrow OH(\{1, 2, 3\})$ and $T_2\colon l^\infty(\{1, \ldots, 4\})\rightarrow OH(\{1, 2\})$ such that
\[ \|T_i\|_{\text{cb}}< C(T_i)\,, \quad i=1, 2\,. \]
In particular, the best constant $C_0$ in the noncommutative little Grothendieck inequality (\ref{eq33343}) is strictly larger than 1.
\end{theorem}

The key result that will be used in the proof of the above theorem is the following:

\begin{theorem}\label{th66767}
Let $(A, \tau)$ be a finite dimensional (unital) C$^*$-algebra with a faithful trace $\tau$\,. Furthermore, let $d$ be a positive integer and let $a_1\,, \ldots\,, a_d$ be elements in $A$ satisfying
\begin{enumerate}
\item [$(i)$] $\tau(a_i^*a_j)=\delta_{ij}$\,, $1\leq i, j\leq d$\,,
\item [$(ii)$] $\sum_{i=1}^d a_i^*a_i=\sum_{i=1}^d a_ia_i^*=d {1_A}$\,.
\end{enumerate}
Consider the map $T\colon A\rightarrow OH(d)\colon=OH(\{1\, \ldots \,, d\})$ given by
\begin{enumerate}
\item [$(iii)$] $Tx\colon =(\tau(a_1^*x)\,, \ldots\,, \tau(a_d^*x))\,, \quad a\in A$\,.
\end{enumerate}
Then $C(T)=1$\,. If, furthermore,
\begin{enumerate}
\item [$(iv)$] $d\geq 2$\,,
\item [$(v)$] $\{a_i^*a_j: 1\leq i, j\leq d\}$ is linearly independent,
\end{enumerate}
then $\|T\|_{\text{cb}}< 1$\,.
\end{theorem}

We will first prove a number of intermediate results.

\begin{lemma}\label{lem667}
Let $A$\,, $\tau$\,, $a_1\,, \ldots\,, a_d$ and $T\colon A\rightarrow OH(d)$ be as in Theorem \ref{th66767} $(i)$\,, $(ii)$ and $(iii)$\,. Then $C(T)=1$\,.
\end{lemma}

\begin{proof}
By $(i)$\,, $a_1\,, \ldots\,, a_d$ is an orthonormal set in $L^2(A, \tau)$\,. Moreover, $\tau(a_1^*x)\,, \ldots\,, \tau(a_d^*x)$ are the coordinates of the orthogonal projection $P$ of $x\in A=L^2(A, \tau)$ onto $E\colon=\text{span}\{a_1\, \ldots\,, a_d\}$ with respect to the basis $\{a_1\,, \ldots\,, a_d\}$\,. Thus
\[ \|Tx\|=\sum\limits_{i=1}^d |\tau(a_i^*x)|^2=\|Px\|_2\leq \|x\|_2={\tau(x^*x)}^{1/2}={\tau(x^*x)}^{1/4}{\tau(x^*x)}^{1/4}\,. \]
Hence $C(T)\leq 1$\,. Conversely, assume that
\[ \|Tx\|\leq K {f_1(xx^*)}^{1/4}{f_2(x^*x)}^{1/4}\,, \quad x\in A\,, \]
for a constant $K> 0$ and states $f_1$\,, $f_2$ on $A$\,. By $(i)$ and $(ii)$\,, it follows that for every $1\leq i\leq d$\,,
\[ T {a_i}=(0\,, \ldots\,, 1\,, 0\,, \ldots\,, 0)\,, \]
where the number 1 is at the $i^{\text{th}}$ coordinate. Therefore,
$1=\|T a_i\|^2\leq K^2 {f_1(a_i {a_i}^*)}^{1/2}{f_2({a_i}^*{a_i})}^{1/2}$\,.
By the Cauchy-Schwarz inequality and $(ii)$\,, we infer that
\[ d=\sum\limits_{i=1}^d \|T a_i\|^2\leq K^2 \left(\sum\limits_{i=1}^2 f_1(a_i a_i^*)\right)^{1/2}\left(\sum\limits_{i=1}^2 f_2(a_i^* a_i)\right)^{1/2}=K^2d\,. \]
Hence $K\geq 1$\,, which proves that $C(T)\geq 1$\,, and the conclusion follows.
\end{proof}

\begin{lemma}\label{lem6678}
Let $A$\,, $\tau$\,, $a_1\,, \ldots\,, a_d$ be as in Theorem \ref{th66767} $(i)$ and $(ii)$\,, set $r\colon= \text{dim}(A)$ and choose $a_{d+1}\,, \ldots\,, a_r$
such that the set $\{a_1\,, \ldots \,, a_r\}$ is an orthonormal basis for $A$\,. Let $B$ be a unital C$^*$-algebra. Then every element $u\in A\otimes B$ has a unique representation of the form
\begin{equation}\label{eq9999998}
u=\sum\limits_{i=1}^r a_i\otimes u_i\,,
\end{equation}
where $u_i\in B$\,, $1\leq i\leq d$\,. Moreover, if $u$ is unitary, then
\begin{equation}\label{eq999999}
\sum\limits_{i=1}^r u_i^*u_i=\sum\limits_{i=1}^r u_i u_i^*=1_B\,.
\end{equation}
\end{lemma}

\begin{proof}
Existence and uniqueness of $u_1\,, \ldots\,, u_r\in B$ in (\ref{eq9999998}) is obvious. To prove (\ref{eq999999})\,, note that if $u^*u=uu^*=1_{A\otimes B}$\,, then by $(i)$\,,
\[ 1_B=(\tau_A\otimes {\text{id}}_B) (u^*u)=\sum\limits_{i, j=1}^d \tau_A(a_i^*a_j)u_i^* u_j=\sum\limits_{i=1}^d u_i^* u_i\,, \]
and similarly, by the trace property of $\tau_A$ and $(i)$\,,
\[ 1_B=(\tau_A\otimes {\text{id}}_B) (uu^*)=\sum\limits_{i, j=1}^d \tau_A(a_i a_j^*)u_i u_j^*=\sum\limits_{i=1}^d u_i u_i^*\,, \]
which completes the proof.
\end{proof}

\begin{lemma}\label{lem667888}
Let $A$\,, $\tau$\,, $a_1\,, \ldots\,, a_d$ and $T$ be as in Theorem \ref{th66767} $(i)$\,, $(ii)$ and $(iii)$\,, and choose $a_{d+1}\,, \ldots\,, a_r$ as in Lemma \ref{lem6678}. Assume further that $\|T\|_{\text{cb}}=1$\,. Then
\begin{enumerate}
\item [$(a)$] For every $\varepsilon> 0$\,, there exists a unital C$^*$-algebra $B(\varepsilon)$, a Hilbert $B(\varepsilon)$-bimodule ${\mathcal H}(\varepsilon)$, elements $u_1\,, \ldots\,, u_r\in B(\varepsilon)$ and unit vectors $\xi\,, \eta\in {\mathcal H}(\varepsilon)$ such that the operator $u\colon =\sum_{i=1}^r a_i\otimes u_i\in A\otimes {B(\varepsilon)}$ is unitary and the following inequalities are satisfied:
    \begin{eqnarray}
    \sum\limits_{i=1}^d \|u_i\xi-\eta{u_i}\|^2+\sum\limits_{i=d+1}^r (\|u_i\xi\|^2+\|\eta{u_i}\|^2)<  \varepsilon\,,\label{eq33223322}\\
    \sum\limits_{i=1}^d \|u_i^*\eta-\xi{u_i^*}\|^2+\sum\limits_{i=d+1}^r (\|u_i^* \eta\|^2+\|\xi u_i^*\|^2)<  \varepsilon\,. \label{eq332233229}
    \end{eqnarray}
\item [$(b)$] There exist a unital C$^*$-algebra $B$, a Hilbert bimodule ${\mathcal H}$, elements $u_1\,, \ldots\,, u_r\in B$ and unit vectors $\xi\,, \eta\in {\mathcal H}$ such that the operator $u\colon =\sum_{i=1}^r a_i\otimes u_i\in A\otimes B$ is unitary and the following identities are satisfied:
\begin{eqnarray}
    &&u_i\xi=\eta{u_i}\,, \quad u_i^*\eta=\xi{u_i^*}\,, \qquad \,\,\, 1\leq i\leq d\,,\label{eq332233228}\\
    && u_i \xi=\eta{u_i}=u_i^* \eta=\xi u_i^*=0\,, \quad d+1\leq i\leq r\,. \label{eq3322332298}
    \end{eqnarray}
\end{enumerate}
\end{lemma}

\begin{proof}
Let $\varepsilon> 0$\,. Since $\|T\|_{\text{cb}}=1$\,, there exists a positive integer $n$ such that $\|T\otimes \text{id}_{M_n(\mathbb{C})}\|_{\text{cb}}> 1-{\varepsilon/4}$\,. Since $\text{dim}(A)< \infty$\,, the unit ball of $M_n(A)=A\otimes M_n(\mathbb{C})$ is the convex hull of its unitary operators. Hence there exists a unitary operator $u\in A\otimes M_n(\mathbb{C})$ such that
\[ \|(T\otimes \text{id}_{M_n(\mathbb{C})})(u)\|_{M_n(OH)}> 1-{\varepsilon/4}\,. \]
By Lemma \ref{lem6678}, $u$ has the form $u=\sum_{i=1}^r a_i\otimes u_i$\, for a unique set of elements $u_1\,, \ldots\,, u_r\in M_n(\mathbb{C})$ satisfying
\begin{equation}\label{eq555555555}
\sum\limits_{i=1}^r u_i^* u_i=\sum\limits_{i=1}^r u_i u_i^*=1_n\,.
\end{equation}
By condition $(iii)$ in Theorem \ref{th66767}\,,
\[ T(a_i)=\left\{\begin{array}{lcl}
                                 e_i\,, \,\,\, 1\leq i\leq d\\
                                 0\,, \,\,\, d+1\leq i\leq r\,,
                                \end{array}
                        \right. \]
where $e_i$ is the $i^{\text{th}}$ vector in the standard unit vector basis of $l^2(\{1\,, \ldots\,, d\})=OH(d)$\,. Hence
\[ (T\otimes \text{id}_{M_n(\mathbb{C})})(u)=\sum\limits_{i=1}^d e_i \otimes {u_i}\,. \]
It then follows (cf. \cite{Pi6}) that
\begin{equation*} \left\|\sum\limits_{i=1}^d u_i \otimes \bar{u_i}\right\|_{M_n(\mathbb{C})\otimes \overline{M_n(\mathbb{C})}}=\left\|\sum\limits_{i=1}^d e_i\otimes u_i\right\|^2_{M_n(OH)}=\left\|(T\otimes \text{id}_{M_n(\mathbb{C})})(u)\right\|_{M_n(OH)}> \left(1-\frac{\varepsilon}{4}\right)^2> 1-\frac{\varepsilon}{2}\,. \end{equation*}
We can identify ${M_n(\mathbb{C})\otimes \overline{M_n(\mathbb{C})}}$ isometrically with the bounded operators on $HS(n)=L^2(M_n(\mathbb{C})\,, \text{Tr})$\,, where $\text{Tr}$ denotes the standard non-normalized trace on $M_n(\mathbb{C})$\,, by letting $a\otimes \bar{b}$ correspond to $L_a R_{b^*}$\,, i.e., $(a\otimes \bar{b})\xi=a\xi b^*$\,, for all $\xi \in HS(n)$\,. In particular,
\[ \sum\limits_{i=1}^d (u_i \otimes \overline{u_i})\xi=\sum\limits_{i=1}^r u_i \xi u_i^*\,, \quad \xi\in HS(n)\,. \]
Let ${HS(n)}_{\text{sa}}$ denote the self-adjoint part of $HS(n)$\,. Since $HS(n)={HS(n)}_{\text{sa}} + i {HS(n)}_{\text{sa}}$\,, and the operator $\sum_{i=1}^d u_i \otimes \overline{u_i} \in {\mathcal B}(HS(n))$ leaves ${HS(n)}_{\text{sa}}$ invariant, one checks easily that the norm of $\sum_{i=1}^d u_i \otimes \overline{u_i}$ is the same as the norm of $\sum_{i=1}^d u_i \otimes \overline{u_i}$ restricted to ${HS(n)}_{\text{sa}}$\,. By compactness of the unit ball in ${HS(n)}_{\text{sa}}$ we deduce that there exist vectors $\xi, \eta \in {HS(n)}_{\text{sa}}$ such that
\begin{equation}\label{eq4444444444444}
\|\xi\|_2=\|\eta\|_2=1\,,
\end{equation}
satisfying, moreover,
\begin{equation*}
\left|\left\langle \sum\limits_{i=1}^d u_i \xi u_i^*\,, \eta \right\rangle _{HS(n)}\right|=\left\|\sum\limits_{i=1}^d u_i \otimes \overline{u_i}\right\|> 1-\frac{\varepsilon}{2}\,.
\end{equation*}
Furthermore, since $\left\langle \sum\limits_{i=1}^d u_i \xi u_i^*\,, \eta \right\rangle _{HS(n)}$ is a real number, we infer that
\begin{equation}\label{eq444444444444456}
\sum\limits_{i=1}^d \langle u_i \xi\,, \eta u_i\rangle =\left\langle \sum\limits_{i=1}^d u_i \xi u_i^*\,, \eta \right\rangle _{HS(n)}> 1-\frac{\varepsilon}{2}\,,
\end{equation}
by replacing $\eta$ with $-\eta$\,, if needed. Hence,
\begin{eqnarray}\label{eq44334434343}
\sum\limits_{i=1}^d \left(\|u_i \xi\|_2^2 + \|\eta u_i\|_2^2 -  \|u_i \xi -\eta u_i\|_2^2\right)&=& 2\text{Re} \sum\limits_{i=1}^d \langle u_i \xi\,, \eta u_i\rangle _{HS(n)}\\
&= & 2 \sum\limits_{i=1}^d \langle u_i \xi u_i^*\,, \eta \rangle_{HS(n)}\,\, > \,\, 2-\varepsilon\,. \nonumber
\end{eqnarray}
Moreover, by (\ref{eq555555555}) and (\ref{eq4444444444444})\,,
\begin{equation}\label{eq2323232111}
\sum\limits_{i=1}^r \|u_i \xi \|_2^2 + \sum\limits_{i=1}^r \|\eta u_i\|_2^2 =2\,.
\end{equation}
Subtracting (\ref{eq2323232111}) from (\ref{eq44334434343})\,, we get
\begin{equation*}
\sum\limits_{i=1}^d \|u_i \xi -\eta u_i\|_2^2 + \sum\limits_{i=d+1}^r \left(\|u_i \xi\|_2^2+ \|\eta u_i\|_2^2\right)< \varepsilon\,.
\end{equation*}
Furthermore, since $\xi=\xi^*$ and $\eta=\eta^*$\,, we get by taking adjoints that
\begin{equation*}
\sum\limits_{i=1}^d \|u_i^* \xi -\eta u_i^*\|_2^2 + \sum\limits_{i=d+1}^r \left(\|u_i \xi\|_2^2+ \|\eta u_i\|_2^2\right)< \varepsilon\,.
\end{equation*}
This proves $(a)$ with $B(\varepsilon)=M_n(\mathbb{C})$ and ${\mathcal H}(\varepsilon)=HS(n)$\,.

We now prove $(b)$\,. Given a positive integer $n$\,, let $\varepsilon_n\colon ={1}/{n^2}$ and set $B_n\colon = B(\varepsilon_n)$ and ${\mathcal H}_n\colon = {\mathcal H}(\varepsilon_n)$\,. Then $B_n$ is a unital C$^*$-algebra and ${\mathcal H}_n$ a $B_n$-Hilbert bimodule. Moreover, there exist elements $u_1^{(n)}\,, \ldots \,, u_r^{(n)}\in B_n$ and unit vectors ${\xi}^{(n)}\,, {\eta}^{(n)}\in {\mathcal H}(n)$ such that the operator $u^{(n)}\colon = \sum_{i=1}^r a_i \otimes u_i^{(n)} \in A\otimes {B_n}$ is unitary and the following inequalities hold:
\begin{equation*}
\left\|u_i^{(n)} \xi_n -\eta_n u_i^{(n)}\right\|< {1}/{n}\,, \quad \left\|\left(u_i^{(n)}\right)^* \eta_n -\xi_n \left(u_i^{(n)}\right)^*\right\|< {1}/{n}\,, \qquad 1\leq i\leq d\end{equation*}
respectively,
\begin{equation*}
\left\|\left(u_i^{(n)}\right)^* \xi_n\right\|< {1}/{n}\,, \quad \left\|\eta_n u_i^{(n)}\right\|< {1}/{n}\,, \quad \left\|\left(u_i^{(n)}\right)^* \eta_n\right\|< {1}/{n}\,, \quad \left\|\xi_n \left(u_i^{(n)}\right)^*\right\|< {1}/{n}\,, \qquad d+1\leq i\leq r\,.
\end{equation*}
Now $(b)$ follows from $(a)$ by a standard ultraproduct construction (see, e.g., \cite{Ha1}).
\end{proof}

\begin{lemma}\label{lem7878788}
Let $A$\,, $\tau$\,, $a_1\,, \ldots\,, a_d$ and $T$ be as in Theorem \ref{th66767} $(i)$\,, $(ii)$ and $(iii)$\,, and assume that $\|T\|_{\text{cb}}=1$\,. Then
\begin{enumerate}
\item [$(a)$] There exist a finite von Neumann algebra $N$ with a normal, faithful tracial state $\tau_N$\,, a projection $p\in N$ and elements $v_1\,, \ldots\,, v_d\in {(1-p)Np}$ such that the operator $v\colon = \sum_{i=1}^d a_i\otimes v_i \in A\otimes N$ is a partial isometry satisfying
    \[ v^*v=1_A\otimes p\,, \quad vv^*=1_A\otimes (1-p)\,. \]
\item [$(b)$] There exist a finite von Neumann algebra $P$ with a normal, faithful tracial state $\tau_P$ and elements $w_1\,, \ldots\,, w_d\in P$ such that the operator $w\colon = \sum_{i=1}^d a_i\otimes w_i\in A\otimes P$ is unitary.
\end{enumerate}
\end{lemma}

\begin{proof}
Let $r\colon = \text{dim}(A)$ and choose $a_1\,, \ldots\,, a_r\in A$ as in Lemma \ref{lem6678}. Further, let $B$\,, ${\mathcal H}$\,, $u_1\,, \ldots\,, u_r$ and $\xi$\,, $\eta$ be as in Lemma \ref{lem667888} $(b)$\,. In particular, the operator $u\colon = \sum_{i=1}^r a_i\otimes u_i$ is a unitary in $A\otimes B$\,.

Note that $M_2({\mathcal H})$ is an $M_2(B)$-bimodule by standard matrix multiplication, and $M_2({\mathcal H})$ is a Hilbert space with norm
\[ \left\|\left(
\begin{array}
[c]{cc}%
\sigma_{11} & \sigma_{12}\\
\sigma_{21} & \sigma_{22}
\end{array}
\right)\right\|^2_{M_2({\mathcal H})}=\sum\limits_{i, j=1}^2 \|\sigma_{ij}\|^2\,, \quad \sigma_{ij}\in {\mathcal H}\,, 1\leq i, j\leq 2\,. \]
Set $\zeta\colon = \frac{1}{\sqrt{2}}\left(
\begin{array}
[c]{cc}%
\xi & 0\\
0 & \eta
\end{array}
\right)\in M_2({\mathcal H})$ and $s_i\colon = \left(
\begin{array}
[c]{cc}%
0 & 0\\
u_i & 0
\end{array}
\right)\in M_2(B)$\,, for all $1\leq i\leq r$\,. Then $\|\zeta\|=1$\,, and by (\ref{eq332233228}) and (\ref{eq3322332298}) it follows that
\begin{eqnarray}
&& s_i \zeta = \zeta s_i\,, \,\, s_i^* \zeta = \zeta s_i^*\,, \quad 1\leq i\leq d\,,\label{eq333333333333333333}\\
&& s_i \zeta =s_i^* \zeta =0\,, \qquad \qquad d+1\leq i\leq r\,.\label{eq3333333333333333334}
\end{eqnarray}
Further, set $e\colon = \left(
\begin{array}
[c]{cc}%
1_B & 0\\
0 & 0
\end{array}
\right)\in M_2(B)$\,. Then $e$ is a projection and by (\ref{eq999999})\,,
\begin{equation}\label{eq6666666666666666}
\sum\limits_{i=1}^r s_i^* s_i=e\,, \quad \sum\limits_{i=1}^r s_i s_i^*= 1_{M_2(B)}-e\,.
\end{equation}
Next, set $s\colon = \sum_{i=1}^r a_i\otimes s_i$\,. Since $u$ is a unitary operator, it follows that
\begin{equation}\label{eq7777777777}
s^* s =1_A\otimes e\,, \quad s s^*= 1_A\otimes (1_{M_2(B)}-e)\,.
\end{equation}
Let $C$ denote the C$^*$-algebra generated by $s_1\,, \ldots\,, s_d$ in $M_2(B)$\,. By (\ref{eq6666666666666666})\,, both $e$ and $1_{M_2(B)}-e$ belong to $C$\,, and hence $1_{M_2(B)}\in C$\,. Moreover, by (\ref{eq333333333333333333}) and (\ref{eq3333333333333333334})\,, \begin{equation}\label{eq99999999999}
c \zeta = \zeta c\,, \quad c\in C\,.
\end{equation}
Define now a state $\phi$ on $C$ by
$\phi(c)\colon = \langle c \zeta\,, \zeta\rangle_{M_2({\mathcal H})}$\,, for all $c\in C$\,. Note that $\phi$ is tracial, since \[ \phi(c^* c)=\langle c \zeta\,, c \zeta\rangle_{M_2({\mathcal H})}=\langle \zeta c\,, \zeta c\rangle_{M_2({\mathcal H})}=\langle \zeta c c^*\,, \zeta\rangle_{M_2({\mathcal H})}=\langle c c^* \zeta\,, \zeta\rangle_{M_2({\mathcal H})}=\phi(c c^*)\,, \quad c\in C\,. \]
Let $(\pi_{\phi}\,, H_{\phi}\,, \xi_{\phi})$ be the GNS-representation of $C$ with respect to $\phi$\,. Then $N\colon = {\pi_{\phi}(C)}^{\prime \prime}$ is a finite von Neumann algebra with normal, faithful tracial state $\tau_N$ given by $\tau_N(x)\colon = \langle x \xi_{\phi}\,, \xi_{\phi}\rangle_{H_{\phi}}$\,, for all $x\in N$\,. Moreover,
$\phi(c)=\tau_N(\pi_{\phi}(c))$\,, for all $c\in C$\,.

Now set $v_i\colon = \pi_{\phi}(s_i)$\,, for all $1\leq i\leq r$\,. By (\ref{eq3333333333333333334}), it follows that $\tau_N(v_i^* v_i)=\phi(s_i^* s_i)=0$\,, for all $d+1\leq i\leq r$\,, and hence
\begin{equation}\label{eq55555555555555555555}
v_i=0\,, \quad d+1\leq i\leq r\,.
\end{equation}
Next set $p\colon = \pi_{\phi}(e)$\,. Then $p$ is a projection in $N$\,. By (\ref{eq6666666666666666}) and (\ref{eq55555555555555555555}) we infer that
\begin{equation}\label{eq555555555555}
\sum\limits_{i=1}^d v_i^* v_i=p\,, \quad \sum\limits_{i=1}^d v_i v_i^*=1_N-p\,.
\end{equation}
Finally, set $v\colon = ({\text{id}_A}\otimes {\pi_{\phi}})(s)=\sum_{i=1}^r a_i\otimes v_i=\sum_{i=1}^d a_i\otimes v_i$\,. Then, by (\ref{eq7777777777}) it follows that
$v^*v =1_A\otimes p$ and $v v^*=1_A\otimes (1_N-p)$\,.
This proves part $(a)$\,.

To prove $(b)$\,, we note first that by (\ref{eq555555555555})\,, $p$ and $1_N-p$ have the same central valued trace, and therefore they are equivalent (as projections in $N$) (see, e.g., \cite{KR} (Vol. II, Chap. 8)). In particular, $\tau_N(p)=\tau_N(1_N-p)=1/2$\,.
Choose now $t\in N$ such that $t^*t=p$ and $tt^*=1_N-p$\,, and set $w_i\colon = t^* v_i$\,, for all $1\leq i\leq d$\,. Then
\[ \sum\limits_{i=1}^d w_i^* w_i=\sum\limits_{i=1}^d w_i w_i^* =p\,, \]
and the operator $w\colon =\sum_{i=1}^d a_i \otimes w_i$ satisfies
\[ w^* w=w w^*=1_A\otimes p\,. \]
Hence $(b)$ follows from $(a)$ by setting $P\colon = pNp$ and defining $\tau_P(x)\colon = 2 \tau_N(x)$\,, for all $x\in P$\,.
\end{proof}

{\bf Proof of Theorem \ref{th66767}}: By (\ref{eq8877887788787}) and Lemma \ref{lem667} we have that $\|T\|_{\text{cb}}\leq C(T)=1$\,. If we assume by contradiction that $\|T\|_{\text{cb}}=1$\,, then by Lemma \ref{lem7878788} $(b)$\,, there exist a finite von Neumann algebra $P$ with a normal, faithful tracial state $\tau_P$ and elements $w_1\,, \ldots\,, w_d\in P$ such that the operator $w\colon = \sum_{i=1}^d a_i\otimes w_i$ is unitary in $A\otimes P$ . By the hypothesis of Theorem \ref{th66767} (cf. $(iv)$ and $(v)$), the additional assumptions that $d\geq 2$ and the set $\{a_i^*a_j: 1\leq i, j\leq d\}$ is linearly independent do hold. Therefore, we can proceed almost as in the proof of Corollary \ref{cor1}. Namely, we have $w^*w=\sum_{i, j=1}^d a_i^* a_j \otimes w_i^* w_j$\,.
Therefore, using $(ii)$ we deduce that
\[ 0_{A\otimes P}=w^* w-1_A\otimes 1_P=\sum\limits_{i, j=1}^d a_i^* a_j \otimes \left(w_i^* w_j -\frac{1}{d} \delta_{ij} 1_P\right)\,. \]
Hence, by $(v)$ we conclude that
\[ w_i^* w_i =\frac{1}{d} \,\delta_{ij} 1_P\,, \quad 1\leq i, j\leq d\,. \]
This implies that $\sqrt{d} w_1$ and $\sqrt{d} w_2$ are two isometries in the finite von Neumann algebra $P$, having orthogonal ranges. This is impossible. Therefore $\|T\|_{\text{cb}}< 1$ and the proof is complete.\qed

{\bf Proof of Theorem \ref{th6000000000006}}:

\noindent $(1)$ Consider $A= M_3(\mathbb{C})$\,, $\tau=\tau_3$\,, $d=3$\,, and
let $a_1\,, a_2\,, a_3\in M_3(\mathbb{C})$ be given by
\[ a_1=\sqrt{\frac{3}{2}}\left(
\begin{array}
[c]{ccc}%
0 & 0 & 0\\
0 & 0 & -1\\
0 & 1 & 0
\end{array}
\right),  \quad a_2=\sqrt{\frac{3}{2}}\left(
\begin{array}
[c]{ccc}%
0 & 0 & 1\\
0 & 0 & 0\\
-1 & 0 & 0
\end{array}
\right)\,, \quad a_3=\sqrt{\frac{3}{{2}}}\left(
\begin{array}
[c]{ccc}%
0 & -1 & 0\\
1 & 0 & 0\\
0 & 0 & 0
\end{array}
\right)\,. \]
Define $T_1\colon M_3(\mathbb{C})\rightarrow OH(3)$ by
\[ T_1(x)\colon = (\tau(a_1^* x)\,, \tau(a_2^* x)\,, \tau(a_3^* x))\,,\quad x\in M_3(\mathbb{C})\,. \]

\noindent $(2)$ Consider $A=l^\infty(\{1, \ldots, 4\})$\,, $\tau(c)=(c_1+\ldots +c_4)/4$\,, $c=(c_1\,, \ldots, c_4)\in A$\,, $d=2$\,, and let $a_1\,, a_2\in A$ be given by
\[ a_1\colon=(\sqrt{2}\,, {\sqrt{2}}/{\sqrt{3}}\,, {\sqrt{2}}/{\sqrt{3}}\,, {\sqrt{2}}/{\sqrt{3}})\,, \quad a_2\colon = (0, 2/{\sqrt{3}}\,, (2/{\sqrt{3}})\omega\,,  (2/{\sqrt{3}})\bar{\omega})\,, \]
where $\omega\colon = e^{{i 2 \pi}/3}$, and $\bar{\omega}$ is its complex conjugate.
Define $T_2\colon l^\infty(\{1, \ldots, 4\})\rightarrow OH(2)$ by
\[ T_2(x)\colon = (\tau(a_1^* x)\,, \tau(a_2^* x))\,,\quad x\in l^\infty(\{1, \ldots, 4\})\,. \]

In each of the cases $(1)$ and $(2)$ it is easily checked that conditions $(i)$\,, $(ii)$\,, $(iv)$ and $(v)$ in the hypothesis of Theorem \ref{th66767} are verified. Hence, the maps $T_i$  (defined above according to condition $(iii)$ of Theorem \ref{th66767}) will satisfy $C(T_i)=1 > \|T_i\|_{\text{cb}}$\,, for $i=1\,, 2$\,.
Note that $a_1\,, a_2\,, a_3$ in case $(1)$ are scalar multiples of the matrices considered in Example \ref{exp1}, while $a_1, a_2$ in case $(2)$ correspond, up to a scalar factor, to the diagonal $4\times 4$ matrices in Example \ref{exp2} with $s=1/3$\,.\qed

\thanks{}

\end{document}